\documentclass[12pt]{amsart}
\usepackage{ShaharStyle}

\title{Bound on multiplicities of symmetric pairs over $p$-adic fields}
\author{Shahar Dagan}
\thanks{Email: \texttt{shahar.dagan@weizmann.ac.il}} 
\begin{document}

\begin{abstract}
    We establish uniform bounds on the multiplicities of irreducible admissible representations appearing in spaces of functions on symmetric spaces over $p$-adic fields. These multiplicities can exceed one and depend intricately on the group, the space, and the representation, making exact computations often difficult to carry out. This motivates the search for bounds depending only on structural invariants of the group and the field.

    More precisely, let $\mathbf{G}$ be a connected reductive group over a $p$-adic field $F$ of large residue characteristic (relative to the rank of $\mathbf{G}$), let $\rho$ be a smooth admissible irreducible representation of $G = \mathbf{G}(F)$ and let $\theta$ be a rational involution with fixed-point subgroup $H=G^\theta$. We show that the multiplicity
    \[
        \dim \Hom_{G}\big(\rho, C^\infty(G/H)\big)
    \]
    is uniformly bounded. The bound depends only on the rank of $\mathbf{G}$ and the residue degree of $F$. Our method relies on results of \cite{offen2017parabolic,hakim2008distinguished,fintzen2021types,aizenbud2024bounds} and combines Mackey's theory with cohomological techniques.

    As an application, we deduce a uniform bound on such multiplicities where the base field $F$ varies over all but a bounded number of local completions of a fixed number field.
\end{abstract}

\maketitle
\section{Introduction}
The main results of this paper are Theorems~\ref{intthm: bound in global settings} and \ref{intthm: multiplicity bound unramified} below.

\begin{introtheorem}[see Theorem~\ref{thm: global bound} below]\label{intthm: bound in global settings}
    Let $L$ be a number field and let $\mathbf{G}$ be a connected reductive group over~$L$.
    There exist constants $C_1, C_2$ such that for any rational involution $\theta$ of $\mathbf{G}$, any finite place $v$ of $L$ with residue characteristic $p_v > C_1$, and any smooth admissible irreducible representation $\rho$ of $\mathbf{G}(L_v)$
    \[
        \dim \Hom_{\mathbf{G}(L_v)}\!\left(
        \rho,\; C^\infty\!\left(\mathbf{G}(L_v) / \mathbf{G}(L_v)^{\theta_v}\right)
        \right) < C_2.
    \]
    Here:
    \begin{itemize}
        \item $L_v$ is the completion of $L$ at $v$;
        \item $\theta_v$ is the involution of $\mathbf{G}(L_v)$ induced by $\theta$;
        \item $\mathbf{G}(L_v)^{\theta_v} \subseteq \mathbf{G}(L_v)$ is the subgroup of $\theta_v$-fixed points.
    \end{itemize}
    Moreover, $C_1, C_2$ depend only on $[L : \bQ]$ and $\rank(\mathbf{G})$.
\end{introtheorem}

Theorem~\ref{intthm: bound in global settings} is deduced from Theorem~\ref{intthm: multiplicity bound unramified} using the fact that reductive groups over local fields with sufficiently large residue characteristic are tamely ramified (see Lemma~\ref{lem: bounded bad primes}), together with uniform bounds on the residue degree and the rank at the finite places of a fixed number field.

\begin{introtheorem}[see Theorem \ref{thm: multiplicity bound unramified} below]\label{intthm: multiplicity bound unramified}
    Let $\mathbf{G}$ be a connected reductive group over a $p$-adic field $F$ with residue field $\ff$ of odd residue characteristic $p$, such that $\mathbf{G}$ splits over a tamely ramified extension of $F$ and $p > \rank(\mathbf{G})+1$. Let $\rho$ be a smooth admissible irreducible representation of $G\coloneqq \mathbf{G}(F)$. Let $\theta$ be a rational involution of $G$, $H = G^\theta$ be the subgroup of $\theta$-fixed points, and $\chi$ be a smooth unramified character of $H$.

    There exists a constant $C$ such that
    \[
        \dim \Hom_G\left( \rho,\, C^\infty(G/H,\chi) \right) < C.
    \]
    Moreover, $C$ depends only on the residue degree $[\ff : \bF_p]$ and the absolute rank $\rank(\mathbf{G})$.
\end{introtheorem}

The assumptions on $\mathbf{G}$ and the residue characteristic are needed in order to apply the construction of all irreducible supercuspidal representations by Yu \cite{yu2001construction,kim2007supercuspidal, fintzen2021types}. Specifically, the condition $p > \rank(\mathbf{G})+1$ is sufficient to deduce that $p$ does not divide the order of the absolute Weyl group of $\mathbf{G}(F)$, which is necessary to apply \cite[Theorem 8.1]{fintzen2021types}.

Theorem~\ref{intthm: multiplicity bound unramified} is a substantial partial verification of \cite[Conjecture C]{aizenbud2024bounds} (see Conjecture~\ref{intconj: bounded multiplicity for spherical} below) in the setting of symmetric spaces.
The proof of Theorem~\ref{intthm: multiplicity bound unramified} hinges on known results on parabolic induction obtained in \cite{offen2017parabolic,Bernstein1977}, on distinguished tame supercuspidal representations obtained in \cite{hakim2008distinguished,fintzen2021types}, and representations of compact subgroups of reductive $p$-adic groups obtained in \cite{aizenbud2024bounds}.

The main contribution of our work is to combine these results to obtain a uniform bound on the multiplicities. This bound depends only on simple structural invariants of the group and the field. Previously, only finiteness was known.

The main tool to move from finite multiplicities to bounded multiplicities is Proposition~\ref{prop: bounded fibers}. It plays a central role in bounding the non-vanishing summands in Mackey's multiplicity formula that arise in the proof.

\begin{proposition}\label{prop: bounded fibers}
    Let $\mathbf{G}$ be a connected reductive group over a $p$-adic field $F$ with residue field $\ff$ of odd characteristic $p$, $\theta$ a rational involution of $\mathbf{G}$, $H\subseteq G \coloneqq \mathbf{G}(F)$ the subgroup of $\theta$-fixed points, and $Q$ a subgroup of $G$ which is either a parabolic subgroup or a compact-mod-center subgroup.

    Then there exists a constant $C$ such that the size of each fiber of the map
    \[
        \iota_Q:Q\backslash G/ H \to Q \backslash G /\theta(Q), \quad
        Q g H \mapsto Q g\theta(g)^{-1} \theta(Q)
    \]
    is bounded by $C$.
    Moreover, $C$ depends only on the residue degree $[\ff : \bF_p]$ and $\rank(\mathbf{G})$.
\end{proposition}

The parabolic case follows from Lemma~\ref{lem: (P,H) double cosets bound} and Remark~\ref{rem: proof of prop: bounded fibers}, and the compact-mod-center case follows from Corollary~\ref{cor: (K,H) double cosets bound}.
\subsection{Background and motivation}

Let $G$ be a group and $X$ a transitive $G$-set. A central problem in representation theory is to describe the decomposition of the space of functions on $X$ into irreducible representations of~$G$, namely, which ones occur and with what multiplicities. Typically, one studies this decomposition within specific classes of functions and representations determined by the structure of $G$ and $X$.

As a basic example, if $G$ is a connected reductive group over $\bC$ and $X$ is a \emph{spherical} $G$-variety (i.e., a Borel subgroup of $G$ has an open orbit on $X$), then the ring of regular functions $\bC[X]$ is multiplicity-free as a $G$-representation.

In other settings, multiplicities may exceed one, and in general they can depend in a subtle way on the specific group $G$, the $G$-space $X$, and the representation $\rho$ of $G$. Determining these multiplicities explicitly often requires a detailed classification of $G$, $X$, and $\rho$, which can be highly intricate and, in many cases, incomplete or unknown. In such situations, direct computation is often out of reach. Nevertheless, one may still hope to obtain more qualitative results, such as a \emph{uniform bound} depending only on a small set of simple invariants, rather than on the full classification of the groups, spaces, and representations involved. This perspective motivates the following conjecture.

\begin{introconjecture}[cf. {\cite[Conjecture C]{aizenbud2024bounds}}]
    \label{intconj: bounded multiplicity for spherical}
    Let $\mathbf{G}$ be a reductive group scheme over $\bZ$, and let $\mathbf{X}$ be a $\mathbf{G}$-scheme. Assume that $\mathbf{X}(\bC)$ is a spherical $\mathbf{G}(\bC)$-space. Then there exists an integer $C$ such that for any local field $F$ and any irreducible admissible representation $\rho$ of $\mathbf{G}(F)$, we have
    \[
        \dim \Hom_{\mathbf{G}(F)}\!\big(\rho, C^\infty(\mathbf{X}(F))\big) < C.
    \]
\end{introconjecture}

This conjecture can be approached at several levels of generality:
\begin{itemize}
    \item When both $F$ and $\rho$ are fixed, many cases have been established; see for example~\cite{van1987asymptotic,delorme2010constant,sakellaridis2017periodsharmonicanalysisspherical}.
    \item The next level asks for uniform bounds in $\rho$ with $F$ fixed. Here, nontrivial results exist mostly in the Archimedean case; see~\cite{van1987asymptotic,Kobayashi_2013,krötz2014finiteorbitdecompositionreal,aizenbud2016holonomicityrelativecharactersapplications}.
    \item The most general form, where both $F$ and $\rho$ vary, is far less understood. Known results in this setting are largely limited to cases where the spherical space is multiplicity-free (Gelfand pair) or to closely related situations; see, e.g.,~\cite{GK,Sha,vD86,Fli91,BvD94,Nie,Yak,AGRS, AGS,OS,AG_AMOT,AG_HC,AG_RegPar,AAG,Zha,JSZ2,JSZ1,AS,AGJ,AG_M1J,SZ,Aiz,CS,Car,Rub}.
\end{itemize}

In this paper we prove a uniform bound for the multiplicity
\[
    \dim \Hom_{G}\!\big(\rho, C^\infty(G/H)\big),
\]
where:
\begin{itemize}
    \item $G=\mathbf G(F)$ is the group of rational points of a connected reductive group $\mathbf{G}$ over a $p$-adic field $F$, such that $\mathbf{G}$ splits over a tamely ramified extension of $F$, and the residue characteristic of $F$ is greater than $\rank(\mathbf{G})+1$;
    \item $\theta$ is a rational involution of $\mathbf{G}$, with $H = G^\theta$ the subgroup of $\theta$-fixed points.
\end{itemize}
The bound depends only on the rank of $\mathbf{G}$ and the residue degree of $F$.

We also deduce a global consequence: if $\mathbf{G}$ is a connected reductive group over a number field $L$ and $\theta$ is a rational involution of $\mathbf{G}$, then for any finite place $v$ of $L$ with sufficiently large residue characteristic (in terms of $[L : \bQ]$ and $\rank(\mathbf{G})$) and any irreducible admissible representation $\rho$ of $\mathbf{G}(L_v)$, the multiplicity
\[
    \dim \Hom_{\mathbf{G}(L_v)}\!\big(\rho, C^\infty(\mathbf{G}(L_v)/\mathbf{H}(L_v))\big)
\]
is uniformly bounded in terms of $[L : \bQ]$ and $\rank(\mathbf{G})$.

\subsection{Overview of the method}

Roughly speaking, our approach is to embed an irreducible representation into a representation obtained via three types of induction from compact subgroups. This allows us to apply Mackey's theory, which expresses the multiplicity as a sum of multiplicities in the compact group setting. Bounds in the compact group setting were established in \cite{aizenbud2024bounds}, by further reducing the question to the setting of finite groups of Lie type, and then applying Lusztig's theory of character sheaves.
Hence our main task is to bound the number of non-vanishing summands arising in Mackey's decomposition, which we achieve using cohomological methods.

More precisely, any irreducible representation appears in the parabolic induction of a supercuspidal representation of a Levi subgroup (see, e.g., \cite{bernstein1976representations}). Furthermore, in \cite{yu2001construction} Yu gave a construction of supercuspidal representations as compact induction from compact-mod-center subgroups. Yu's construction is known to be exhaustive in many cases as shown in \cite{kim2007supercuspidal,fintzen2021types}. In addition, as we show in \S\ref{subsec: compact-mod-center to compact}, any representation of a compact-mod-center subgroup embeds in an induction from a subgroup of bounded index, which is a product of a compact subgroup and a central torus. Finally, a representation of such a group is determined by a representation of the compact subgroup together with a compatible central character. The space of smooth functions on the symmetric space may also be realized as an induced representation:
\[
    C^\infty(G/H,\chi) = \Ind_H^G\chi.
\]

All of these inductive constructions allow us to apply Mackey's theory. Roughly speaking, given two subgroups $L, R \subseteq G$, and two representations $\lambda$ and $\rho$ of $L$ and $R$ respectively, Mackey's decomposition formula states that
\[
    \dim \Hom_G(\ind^G_L \lambda, \ind^G_R \rho)
    = \sum_{x \in L \backslash G / R}
    \dim \Hom_{L \cap \operatorname{int}_{x}(R)}(\lambda, \rho \circ \operatorname{int}_{x^{-1}}),
\]
where $\operatorname{int}_{x}(g) = xgx^{-1}$.

This decomposition guides our analysis along two main axes:
\begin{enumerate}
    \item bounding the number of double cosets that contribute non-trivially;
    \item controlling the multiplicities in the resulting twisted cases.
\end{enumerate}

The use of Mackey's theory in the context of symmetric pairs over a $p$-adic field has been previously studied: the parabolic induction is treated in \cite{offen2017parabolic}, while compact induction arising from Yu's construction is analyzed in \cite{hakim2008distinguished}.

Both papers give a similar geometric description of the double cosets that contribute non-trivially to the sum in Mackey's decomposition, and describe the summands as the multiplicity of some other symmetric pairs (see Propositions~\ref{prop: offen geometric lemma for cuspidal} and \ref{prop: cuspidal to compact} for a more formal statement).

The geometric description of the contributing double cosets is given in terms of the fiber of the map appearing in Proposition~\ref{prop: bounded fibers}. This fiber can be identified with the kernel of a map between group cohomology of $C_2 = \langle \theta \rangle$ with coefficients in $\theta$-stable subgroups of $G$ (see Lemma~\ref{lem: fiber to cohomology}).

This identification applies because we work in the setting of symmetric pairs, and follows the framework developed in \cite{delorme2006analogue}. In \S\ref{subsec: bounds for H1(S2,K)}, we bound the size of the cohomology sets in terms of the rank of the group $\mathbf G$ and the residue degree of the field $F$, using general bounds on group cohomology, building on results from \cite{aizenbud2024bounds}.

In summary, we reduce the multiplicity in Theorem~\ref{intthm: multiplicity bound unramified} to a sum of a bounded number of compact symmetric pair multiplicities, each of which is bounded by~\cite[Corollary~B]{aizenbud2024bounds}. Our method for the proof of Theorem~\ref{intthm: multiplicity bound unramified} is reflected in the following diagram:

\begin{center}
    \begin{tikzpicture}[
            box/.style={draw, minimum height=1.3cm, align=center},
            note/.style={align=left},
            arrow/.style={-Latex, thick},
        ]
        \node[box] (irrepG) {$\opr{Irrep}(G)$};
        \node[box, right=3cm of irrepG] (cuspM) {$\opr{Cusp}(M)$ \\ $M\subseteq G \text { Levi}$};
        \node[box, right=3cm of cuspM] (irrepK) {$\opr{Irrep}(K)$ \\ $K\subseteq M \text { compact-mod-center }$};
        \node[box, below=1.5cm of irrepK] (irrepK') {$\opr{Irrep}(K')$ \\ $K'\subseteq K \text { compact}$};
        \node[box, left=3cm of irrepK'] (irrepFp) {$\opr{Irrep}\,\!\bigl(\cG(\mathbb{F}_p)\bigr)$\\   $\cG(\mathbb{F}_p)$ finite group of Lie type};

        \draw[arrow] (irrepG) -- node[midway, above=6mm, align=left]
        {\cite{offen2017parabolic} (\emph{geometric lemma}) \\+ cohomological bounds} (cuspM);

        \draw[arrow] (cuspM) -- node[midway, above=6mm, align=left]
        {\cite{hakim2008distinguished} (Yu's construction) \\+ cohomological bounds} (irrepK);

        \draw[arrow] (irrepK) -- node[midway, right=2mm, align=left]
        {\S\ref{subsec: compact-mod-center to compact} $p$-adic groups\\ structure} (irrepK');

        \draw[arrow] (irrepK') -- node[midway, above=2mm, align=left]
        {\cite{aizenbud2024bounds}} (irrepFp);

        \node[note, below=1mm of irrepFp]
        {\cite{aizenbud2024bounds,She} (Lusztig's character sheaves)};
    \end{tikzpicture}
\end{center}

In the global setting, we fix a number field $L$ and a reductive group $\mathbf G$ over $L$. For all but a bounded number of finite places $v$ of $L$, Yu's construction exhausts all supercuspidal representations of the group $G(L_v)$ (\cite[Theorem 8.1]{fintzen2021types}). Thus we may apply our local result of Theorem~\ref{intthm: multiplicity bound unramified}. Moreover, for such places, both the rank of $\mathbf G_{L_v}$ and the residue degree of $L_v$ are bounded in terms of $[L : \bQ]$ and $\rank(\mathbf G)$ (Lemmas~\ref{lem: residue field degree} and~\ref{lem: rank descent}), yielding Theorem~\ref{intthm: bound in global settings}.
\subsection{Possible extensions}\label{subsec: possible extensions}

While the bounds we obtain are largely uniform, they still depend on the residue degree of the local field $F$. This dependence is inherited from the methods of \cite{aizenbud2024bounds}, on which our argument builds, though it may eventually be overcome.

In our work, this dependence arises primarily from the use of general cohomological bounds for finite groups, as provided in \cite[Proposition~6.0.5]{aizenbud2024bounds}. A more refined analysis could yield stronger bounds, and might even remove the dependence on residue degree altogether. Such analysis may be carried out in one of the following ways:
\begin{enumerate}
    \item sharpen the cohomological estimates in \cite[\S6]{aizenbud2024bounds};
    \item exploit the specific structure of compact subgroups via Bruhat--Tits theory;
    \item  exploit the particular representations relevant to our setting.
\end{enumerate}
See Remarks~\ref{rem: better bounds for coho of compact} and \ref{rem: M-admissible refinement} for a more detailed discussion of possible strategies along these lines.

In the global setting, the assumption in Theorem~\ref{intthm: bound in global settings} on the residue characteristic being large enough could potentially be removed by establishing uniform bounds also at the remaining places. The finiteness of multiplicities for a fixed representation at such places was shown in~\cite{delorme2010constant}, but the method of \cite{delorme2010constant} depends on the depth of the representation, and thus does not appear to yield a uniform bound. For Archimedean places, a uniform bound is available;  see \cite{Kobayashi_2013, krötz2014finiteorbitdecompositionreal, aizenbud2016holonomicityrelativecharactersapplications}.
\subsection{Structure of the paper}
The paper is organized as follows.
\begin{itemize}

    \item In \S\ref{sec: preliminaries and notation}, we fix notation and basic assumptions on the groups and representations involved. Additionally, we introduce some elementary definitions and results on group cohomology and groups over number fields and ramification.

    \item In \S\ref{sec: Mackey's decomposition}, we apply Mackey's theory to parabolic induction, compact induction from compact-mod-center subgroups, and finite induction from compact times central torus subgroups.

    \item In \S\ref{sec: double cosets counts}, we bound the number of relevant double cosets in Mackey's decomposition using group cohomology.

    \item In \S\ref{sec: proof of main results}, we assemble the previous results to prove Theorems~\ref{intthm: bound in global settings} and \ref{intthm: multiplicity bound unramified}.
\end{itemize}
\subsection{Acknowledgments}
I would like to thank my advisor, \textbf{Avraham Aizenbud}, for guiding me in this project. I would also like to thank \textbf{Guy Shtotland} for helpful discussions.

I was partially supported by ISF grant 1781/23 and BSF grant 2022193.

\section{Preliminaries and Notation}\label{sec: preliminaries and notation}

\subsection{\texorpdfstring{$p$}{p}-adic Representations and Symmetric Spaces}\label{subsec: p-adic Representations and Symmetric Spaces}
Unless stated otherwise, we use the following conventions throughout the paper:
\begin{enumerate}
    \item $F$ is a $p$-adic field, that is, a non-Archimedean local field of characteristic $0$. We denote the separable closure by $F^{\opr{sep}}$, the ring of integers by $\cO_F$, the uniformizer by $\varpi$, the residue field by $\ff$, the residue characteristic by $p$ (which is often assumed to be odd), and the residue degree by $f_F = [\ff:\bF_p]$;
    \item\label{item: G} $\mathbf{G}$ is a connected reductive group over $F$, and $G\coloneqq \mathbf{G}(F)$ its set of $F$-points. For a field extension $E/F$ we let $\mathbf G_{E}$ be the base change of $\mathbf  G$ to $E$;
    \item $\mathbf{G}_{\opr{der}}$ is the derived subgroup of $\mathbf{G}$, and $\mathbf{Z}$ is the maximal central torus of $\mathbf{G}$. We also denote by $G_{\opr{der}} \coloneqq \mathbf{G}_{\opr{der}}(F)$ and $Z \coloneqq \mathbf{Z}(F)$ their corresponding groups of points.
    \item $\rank (\mathbf G)$ is the absolute rank of $\mathbf G$, i.e., if $\mathbf T$ is a maximal torus of $\mathbf G$, and $\mathbf T_{F^{\opr{sep}}}\cong \bG_m^{\times n}$ is its split form over $F^{\opr{sep}}$, then $\rank (\mathbf G) = n$;
    \item $\theta\in \Aut(\mathbf G)$ is a rational involution. We denote $\langle \theta \rangle\cong C_2$ the cyclic group of order $2$ generated by $\theta$, which acts on $G$ via $\theta$;
    \item For any elements $x,g\in G$ and any subgroup $Q\subseteq G$
          \begin{itemize}
              \item $\opr{int}_g(x)\coloneqq gxg^{-1}$;
              \item $\iota(g)\coloneqq g\theta(g)^{-1}$, called the symmetrization map.\footnote{The name comes from the involution $\theta(x)=x^{-T}$ on $\GL_n$.} Note that $\theta(\iota(g))=\iota(g)^{-1}$ and that $\iota$ induces a map $\iota_Q: Q\backslash G /H \to Q \backslash G /\theta(Q)$;
              \item $g.\theta(x) = g\theta(g^{-1}xg)g^{-1}$,
                    and $g.\iota(x) = x\cdot g.\theta(x)^{-1}$;
              \item $Q(x) \coloneqq Q\cap x.\theta(Q)$ and $Q^{x.\theta} \coloneqq Q\cap \opr{int}_x(H) = \{q\in Q | x.\theta(q)=q\}$.
          \end{itemize}
          \textbf{Warning:} This notation differs from \cite{offen2017parabolic}, where  $ Q(x) = Q\cap \opr{int}_x\theta(Q)$, which is only defined when $\theta(x)=x^{-1}$ and is sometimes projected to a Levi subgroup, depending on context.
    \item\label{item: H and chi} $H=G^\theta\subseteq G$ is the subgroup of $\theta$-fixed elements, and $\chi$ is a smooth character of $H$;
    \item\label{notation: standard} (cf. \cite[\S2.2]{offen2017parabolic}) We fix a minimal parabolic subgroup $\mathbf{P_0}$ of $\mathbf{G}$ and a $\theta$-stable maximal split torus $\mathbf{S}$ of $\mathbf{G}$ contained in $\mathbf{P_0}$ (see \cite[Lemma 2.3]{helminck1993rationality}). Let $\mathbf{M_0} = \mathbf{Z_G(S)}$ (the centralizer of $\mathbf{S}$ in $\mathbf{G}$ as an algebraic group) and $\mathbf{U_0}$ the unipotent radical of $\mathbf{P_0}$. Then $\mathbf{P_0} = \mathbf{M_0} \ltimes \mathbf{U_0}$ is a $\theta$-stable Levi decomposition.

          We call a parabolic subgroup $\mathbf{P}$ of $\mathbf{G}$ standard if it contains $\mathbf{P_0}$. If $\mathbf{P}$ is a standard parabolic subgroup of $\mathbf{G}$ then
          it contains a unique Levi subgroup $\mathbf{M}$ containing $\mathbf{M_0}$. The group $\mathbf{M}$ is then called a standard Levi subgroup. If $\mathbf{U}$ is the unipotent radical of $\mathbf{P}$ we say that $\mathbf{P} = \mathbf{M} \ltimes \mathbf{U}$ is a
          standard Levi decomposition.

          This terminology carries over to groups of $F$-points. Namely, a subgroup $P\subseteq G$ ($M\subseteq G$) is called standard parabolic (Levi) subgroup if $P = \mathbf{P}(F)$ ($M = \mathbf{M}(F)$) for a standard parabolic (Levi) subgroup  $\mathbf{P}\subseteq \mathbf{G}$ ($\mathbf{M}\subseteq \mathbf{G}$).
          Note that as $S\coloneqq\mathbf{S}(F)$ is Zariski-dense in $\mathbf{S}$, we have $\mathbf{M_0}(F)=Z_{G}(S)$.

          Any Levi decomposition is conjugate to a standard one (possibly by different elements for the various subgroups).
    \item We use standard notation for functors in the category of smooth representations of $p$-adic groups:
          \begin{itemize}
              \item $\Ind_H^G \sigma$ $(\ind_H^G \sigma)$: (compact) induction of $\sigma$, i.e., the space of (compactly supported) locally constant functions $f : G \to V$ satisfying $f(hg) = \sigma(h)f(g)$, with $G$ acting by right translation. We write $C^\infty(G/H,\chi)\coloneqq \Ind_H^G \chi$ for a character $\chi$ of $H$, and $C^\infty(G/H)$ when $\chi$ is trivial;
              \item $\Res_H^G \pi $ or $\pi|_{H}$: restriction of a representation $\pi$ of $G$ to $H$;
              \item $r_{M,G} (\pi) $: normalized Jacquet module of a representation $\pi$ of $G$ with respect to a parabolic subgroup $P \subseteq G$ and its Levi component $M$;
              \item $i_{M,G}(\sigma) $: normalized parabolic induction of a representation $\sigma$ of the Levi component $M$ of a parabolic subgroup $P \subseteq G$.
          \end{itemize}

          All representations are assumed to be smooth, admissible, and complex-valued.
    \item If $\rho,\sigma$ are representations of a group $Q$, set
          \[
              \langle \rho, \sigma \rangle_Q
              \coloneqq \dim \Hom_Q(\rho,\sigma),
          \]
          and refer to it as the multiplicity of $\rho$ in $\sigma$.
    \item\label{item: K notation} $K$ is a compact-mod-center subgroup of $G$ or $M$, depending on context. We always assume that $K$ contains the center of the ambient group.
    \item\label{item: pi notation} $\pi$ is an irreducible representation of $K$.
\end{enumerate}
In our setting, $K$ and $\pi$ arise from Yu's construction (see \cite{hakim2008distinguished} for the required details). In particular, we have that $\ind_K^G\pi$ is irreducible, which explains why we may assume $Z(G)\subseteq K$.

For the purposes of Theorem \ref{intthm: multiplicity bound unramified}, we assume the character $\chi$ of $H$ is unramified, in the sense defined below.
\begin{definition}\label{def: unramified charcter}
    A smooth character $\chi: H \to \bC^\times$ is \emph{unramified} if it is trivial on every compact subgroup of $H$.
\end{definition}
\subsection{Group Cohomology}\label{subsec: group cohomology}
Group cohomology measures the failure of the fixed-point functor to be exact (see Lemma~\ref{lem: exact sequence of coho} below for a more precise statement).
In this paper we follow the exposition of group cohomology as in \cite{serre1979galois}. We briefly recall the definition and some basic results of group cohomology. Later, in \S\ref{subsec: compact-mod-center to compact} and \S\ref{sec: double cosets counts}, we specialize to two specific cases:
\begin{enumerate}
    \item  the acting group is the Galois group of some field, also known as Galois cohomology;
    \item  the acting group is $C_2$, that is, the action comes from an involution.
\end{enumerate}

\begin{definition}[cf. {\cite[\S I5.1]{serre1979galois}}]\label{def: group cohomology}
    Let $A$ be a group equipped with an action of another group $\Gamma$. Define the set of \emph{$1$-cocycles}:
    \[
        Z^1(\Gamma,A) := \big\{
        f: \Gamma \to A
        \,\big|\,
        f(gh) = f(g)\cdot g.f(h)
        \big\}.
    \]
    Two cocycles $f$ and $f'$ are said to be \emph{cohomologous} if there exists $a\in A$ such that $f'(g) = a^{-1} \cdot f(g) \cdot g.a$.
    This is an equivalence relation in $Z^1(\Gamma,A)$, and the quotient set is denoted $H^1(\Gamma,A)$. This is the \emph{first cohomology set of $\Gamma$ in $A$}.
\end{definition}

We do not require $A$ to be abelian. In this general setting, $H^1(\Gamma,A)$ is referred to as \emph{non-abelian group cohomology}. While the standard theory often assumes $A$ is a $\Gamma$-module (and thus abelian), several aspects can be extended to the non-abelian case.

The following two lemmas allow us to bound the group cohomology in terms of simpler cohomology sets.
\begin{lemma}[cf. {\cite[Propositions 36 and 38]{serre1979galois}}]
    \label{lem: exact sequence of coho}
    Let $A$ be a group equipped with an action of another group $\Gamma$. Let $A' \subseteq A$ be a $\Gamma$-stable normal subgroup. Then there is an exact sequence of pointed sets:
    \[
        \label{eq: exact sequence of coho}
        1 \to (A')^\Gamma
        \to A^\Gamma
        \to (A/A')^\Gamma
        \to H^1(\Gamma, A')
        \to H^1(\Gamma, A)
        \to H^1(\Gamma, A/A'),
    \]
    where $(-)^\Gamma$ is the functor of taking $\Gamma$-fixed groups.

    In particular,
    \[
        |H^1(\Gamma, A)| \leq
        |H^1(\Gamma, A')| \cdot
        |H^1(\Gamma, A/A')|.
    \]

    Moreover, if $A'$ is not normal in $A$, the above sequence remains exact up to the final term, that is, omitting $H^1(\Gamma, A/A')$.
\end{lemma}

\begin{lemma}[cf. {\cite[\S I2.6b)]{serre1979galois}}]\label{lem: inf-res exact seq}
    Let $\Gamma$ be a group, $N$ be a closed normal subgroup of $\Gamma$, and $A$ be a $\Gamma$-module. We have the following exact sequence
    \[
        0 \to
        H^1(\Gamma/N,A^N) \to
        H^1(\Gamma,A) \to
        H^1(N,A)^{\Gamma/N}.
    \]
\end{lemma}

\begin{lemma}\label{lem: vanishing of C_p coho}
    Let $p$ be a prime number, and let $C_p$ be the cyclic group of order $p$. For any finite $C_p$-module $A$ of size $n$ such that $p \nmid n$, we have
    \[
        |H^1(C_p,A)| = 1.
    \]
\end{lemma}
\begin{proof}
    Let $\gamma$ be a generator for $C_p$. As $p\nmid n$, we may define the map
    \[
        \phi: A \to A^{C_p},\qquad a\mapsto \frac{1}{p}\sum_{i=0}^{p-1} \gamma^i.a,
    \]
    which is surjective and $C_p$-invariant. By Lemma~\ref{lem: exact sequence of coho} we have
    \[
        |H^1(C_p, A)| \leq
        |H^1(C_p, A^{C_p})| \cdot
        |H^1(C_p, \ker (\phi))|.
    \]

    We have
    \[
        |H^1(C_p, A^{C_p})| = |\Hom(C_p, A^{C_p})| =1,
    \]
    and
    \[
        |H^1(C_p, \ker (\phi))| = |\ker (\phi)/((\gamma-1)\ker (\phi))|=1,
    \]
    where the second equality comes from the fact that $(\gamma-1)|_{\ker (\phi)}$ is injective as the only element in $\ker (\phi)$ fixed by $\gamma$ is $0$.
\end{proof}

\begin{lemma}\label{lem: vanishing of l neq p coho}
    Let $\Gamma$ be a pro-$p$ group and let $A$ be a finite $\Gamma$-module of size $n$. If $p \nmid n$ then
    \[
        |H^1(\Gamma,A)| = 1.
    \]
\end{lemma}
\begin{proof}
    Let $N\subseteq \Gamma$ be the kernel of the action of $\Gamma$ on $A$, and $\Gamma'\coloneqq \Gamma/N$.  As $N$ acts trivially on $A$, $N$ is pro-$p$ and $p\nmid |A|$, we have
    \[
        H^1(N,A) = \Hom(N,A) = 0.
    \]
    By Lemma~\ref{lem: inf-res exact seq} we have
    \[
        H^1(\Gamma,A) \cong H^1(\Gamma/N,A^N) = H^1(\Gamma',A).
    \]

    As $A$ is finite, $N$ is an open pro-$p$ normal subgroup and $\Gamma'$ is a finite $p$-group. Therefore, $\Gamma'$ has a finite composition series with quotients $C_p$. The result follows by iterating Lemma~\ref{lem: inf-res exact seq}, and then invoking Lemma~\ref{lem: vanishing of C_p coho}.
\end{proof}
\subsection{Groups Over Number Fields and Ramification}
We also fix the following conventions throughout the paper:
\begin{itemize}
    \item $L$ is a number field with ring of integers $\cO_L$;
    \item $v$ is a place of $L$, $L_v$ is the completion of $L$ at $v$, $\fl_v$ is the residue field of $L_v$,  $p_v$ is the characteristic of $\fl_v$, and $f_v \coloneqq [\fl_v:\bF_{p_v}]$ is the residue degree;
    \item $\mathbf{G}$ is a connected reductive group over $L$ equipped with a rational involution $\theta$;
    \item $\theta_v$ is the involution of $\mathbf G(L_v)$ induced by $\theta$;
    \item  $\mathbf{G}(L_v)^{\theta_v} \subseteq \mathbf{G}(L_v)$ is the subgroup of $\theta$-fixed points.
\end{itemize}

We recall some basic notions from number theory used throughout the paper.

\begin{definition}[cf. {\cite[1, \S4]{serre2013local}}]
    \label{def: ramification}
    Let $F'/F$ be a finite extension of $p$-adic fields of residue characteristic $p$. Let $\cO_{F'}$ and $\cO_F$ be their respective rings of integers, and let $\fp_{F'} \subset \cO_{F'}$ and $\fp_F \subset \cO_F$ be the corresponding maximal ideals.

    The \emph{ramification index} $e$ of the extension $F'/F$ is the unique integer such that
    \[
        \fp_F \cO_{F'} = \fp_{F'}^e.
    \]
    The extension $F'/F$ is called \emph{tamely ramified} if $p \nmid e$, and \emph{wildly ramified} otherwise.
\end{definition}

\begin{definition}
    A connected reductive group $\mathbf G$ over a $p$-adic field $F$ is \emph{tamely ramified} if there exists a tamely ramified extension $F'/F$ such that $\mathbf G_{F'}$ is split; otherwise it is called \emph{wildly ramified}. We use the same terminology for the group of rational points $G\coloneqq \mathbf G(F)$.
\end{definition}

The main reason we work with tamely ramified reductive groups is to invoke \cite[Theorem 8.1]{fintzen2021types}, which states that all supercuspidal representations of such a group are obtained by Yu's construction \cite{yu2001construction}, given that the residue characteristic of $F$ does not divide the Weyl group of $G$.
In this paper we assume that the residue characteristic is greater than the rank of the group plus one, which is a slightly stronger assumption on the residue characteristic.

Consequently, we are able to use the results of \cite{hakim2008distinguished} on these representations in the setting of symmetric spaces.

We prove some elementary results on tamely ramified reductive groups.

\begin{lemma}\label{lem: levi of tamely ramified is tamely ramified}
    Let $\mathbf G$ be a tamely ramified connected reductive group, and let $\mathbf M\subseteq \mathbf G$ be a standard Levi subgroup. Then $\mathbf M$ is also tamely ramified.
\end{lemma}
\begin{proof}
    We have the standard Levi decomposition $\mathbf P=\mathbf M\ltimes \mathbf U$. Let $F'/F$ be a tamely ramified extension such that $\mathbf G_{F'}$ is split. $\mathbf P_{F'}$ is a parabolic subgroup of $\mathbf G_{F'}$ and hence contains a maximal split torus. There exists a Levi subgroup of $\mathbf P_{F'}$ containing this torus \cite[Corollary 8.4.4]{springer1998linear}, and as all Levi subgroups of a fixed parabolic are conjugate, $\mathbf M_{F'}$ also contains a maximal split torus.
\end{proof}

\begin{lemma}
    \label{lem: bounded bad primes}
    Fix positive integers $d,r$. There exists a constant $C = C(d,r)$ such that for every
    \begin{itemize}
        \item number field $L$ with $[L : \bQ] \leq d$;
        \item connected reductive group $\mathbf G$ over $L$ with $\rank(\mathbf G) \leq r$;
    \end{itemize}
    the number of finite places $v$ of $L$ for which $\mathbf G_{L_v}$ is wildly ramified is at most $C$.
\end{lemma}

\begin{proof}
    By \cite[Lemma B.1.1]{aizenbud2024bounds}, there exists a function $C^{\opr{spt}} : \bN \to \bN$ such that any reductive group over a field $F$ splits over an extension $F'/F$ of degree at most $C^{\opr{spt}}(\rank(\mathbf G))$. Hence, it suffices to show that for all but a bounded number of finite places $v$, every field extension $L'_v/L_v$ of degree at most $C^{\opr{spt}}(\rank(\mathbf G))$ is tamely ramified.

    Let $v$ be a place of $L$, and let $L'_v/L_v$ be a field extension of degree at most $C^{\opr{spt}}(\rank(\mathbf G))$. Let $e$ denote the ramification index of $L'_v/L_v$. By \cite[II, \S2, Cor.~1]{serre2013local}, $e$ divides $[L'_v : L_v] \leq C^{\opr{spt}}(\rank(\mathbf G))$. If the residue characteristic $p_v$ of $L_v$ satisfies $p_v > C^{\opr{spt}}(\rank(\mathbf G))$, then $p_v \nmid e$, and thus $L'_v/L_v$ is tamely ramified.

    It remains to bound the number of places $v$ for which $p_v \leq C^{\opr{spt}}(\rank(\mathbf G))$. For a fixed prime $p$, each such place corresponds to a prime ideal $\fp$ of $\cO_L$ lying over $(p) \subset \bZ$. By \cite[I, \S4, Prop.~10]{serre2013local}, the sum of the ramification indices over $(p)$ is at most $[L : \bQ]$. Hence, the number of such places is at most $[L : \bQ] \cdot C^{\opr{spt}}(\rank(\mathbf G))$, and the total number of places with $p_v \leq C^{\opr{spt}}(\rank(\mathbf G))$ is bounded as claimed.
\end{proof}

We next verify that the rank and the size of the residue field can be controlled globally in terms of the number field $L$.

\begin{lemma}
    \label{lem: residue field degree}
    Let $L$ be a number field, and let $v$ be a finite place of $L$ with residue degree $f_v$. Then
    \[
        f_v \leq [L : \bQ].
    \]
\end{lemma}

\begin{proof}
    By \cite[II, \S2, Corollary 1]{serre2013local}, the residue degree $f_v$ and the ramification index $e_v$ satisfy
    \[
        f_v \cdot e_v = [L_v : \bQ_p].
    \]
    Let $p$ be a prime number such that $v \mid p$. By \cite[II, \S1, Corollary 3]{lang1994algebraic}, we have
    \[
        L \otimes_{\bQ} \bQ_p \cong \prod_{u \mid p} L_u, \quad \text{with} \quad \sum_{u \mid p} [L_u : \bQ_p] = [L : \bQ].
    \]
    In particular, $[L_v : \bQ_p]$ is less than $[L : \bQ]$, and therefore
    \[
        f_v \leq [L_v : \bQ_p] \leq [L : \bQ].
    \]
\end{proof}

\begin{lemma}
    \label{lem: rank descent}
    Let $\mathbf G$ be a reductive group over a number field $L$, and let $v$ be a finite place of $L$. Then
    \[
        \rank(\mathbf G) = \rank(\mathbf G_{L_v}).
    \]
\end{lemma}

\begin{proof}
    Since both $\mathbf G_{L^{\opr{sep}}}$ and $\mathbf G_{L_v^{\opr{sep}}}$ are split, and because
    \[
        \mathbf G_{L^{\opr{sep}}} \times_{L^{\opr{sep}}} L_v^{\opr{sep}} \cong \mathbf G_{L_v^{\opr{sep}}},
    \]
    it follows that the ranks of $\mathbf G$ over $L$ and over $L_v$ agree.
\end{proof}

\section{Mackey's Decompositions}\label{sec: Mackey's decomposition}
The proof of Theorem~\ref{intthm: multiplicity bound unramified} involves three types of induction --- parabolic induction from a supercuspidal representation, compact induction from a compact-mod-center subgroup arising from Yu's construction (see \S\ref{subsec: induct from compact}), and finite induction from subgroups of the form compact subgroup times a maximal central torus.

In the setting of symmetric pairs over $p$-adic groups, the parabolic induction was studied by Offen in \cite{offen2017parabolic} and the compact induction from a compact-mod-center subgroup was studied by Hakim and Murnaghan in \cite{hakim2008distinguished}. Both used Mackey's theory, and achieved a geometric description of the non-vanishing summand in Mackey's decomposition formula. In \S\ref{subsec: parabolic induction} and \S\ref{subsec: induct from compact}, we briefly recall their results.

In \S\ref{subsec: compact-mod-center to compact} we also apply Mackey's decomposition to bound the multiplicities arising in the compact-mod-center setting, by reducing to the compact setting.
\subsection{Parabolic induction}\label{subsec: parabolic induction}
Let $\mathbf G$ be a connected reductive group over a $p$-adic field $F$, equipped with a rational involution $\theta$. Let $G=\mathbf G(F)$ be the group of $F$-points, $H\coloneqq G^\theta$  be the subgroup of $\theta$-fixed elements, and $\chi$ be a smooth character of $H$. Let $P \subseteq G$ be a standard parabolic subgroup with standard Levi decomposition $P = M U$, and let $\sigma$ be a representation of $M$.

In \cite{offen2017parabolic}, Offen studies when the parabolically induced representation $i_{M,G}(\sigma)$ is $(H, \chi)$-distinguished, that is,
\[
    \langle i_{M,G}(\sigma), C^\infty(G/H, \chi) \rangle_G \neq 0.
\]
In particular, in \S4–5 of \cite{offen2017parabolic} Offen proves the following formula, under the additional assumption that $\sigma$ is supercuspidal.

\begin{proposition}[cf.~{\cite[Prop.~4.1 and Cor.~5.2]{offen2017parabolic}}]
    \label{prop: offen geometric lemma for cuspidal}
    Let $G,\theta,H,P,M,\sigma,\chi$ be as above, and assume that $\sigma$ is supercuspidal. Then there exist:
    \begin{itemize}
        \item a finite set $\cW$ of representatives for the $(P,H)$-double cosets satisfying
              \[
                  \eta.\theta(M)=M;
              \]
        \item for each $\eta \in \cW$, an unramified smooth character\footnote{This character arises from the normalization of parabolic induction.} $\Delta_\eta$ of
              \[
                  M^{\eta.\theta} \coloneqq \{m \in M \mid \eta.\theta(m)=m\};
              \]
    \end{itemize}
    such that
    \[
        \langle i_{M,G}(\sigma), C^\infty(G/H, \chi) \rangle_G
        =
        \sum_{\eta \in \cW}
        \left\langle
        \sigma,\,
        C^\infty\!\left(
        M / M^{\eta.\theta},\,
        \Delta_\eta \cdot \chi \circ \operatorname{int}_{\eta^{-1}}
        \right)
        \right\rangle_M .
    \]
\end{proposition}

In \cite[\S5]{offen2017parabolic}, Offen refers to the elements of $\cW$ as \emph{$M$-admissible}, and further studies their properties.
\subsection{Induction from compact-mod-center subgroups}\label{subsec: induct from compact}
In this work, we consider only supercuspidal representations arising via a construction by Yu \cite{yu2001construction}, which was shown to be exhaustive in many cases \cite{kim2007supercuspidal,fintzen2021types}. We refer to them as tame supercuspidal representations.
In particular, Yu's construction is a method to build an irreducible supercuspidal representation as the compact induction from a compact-mod-center subgroup associated to a ``Yu-datum'' (see \cite[\S3]{hakim2008distinguished} for more details).

In \cite{hakim2008distinguished}, Hakim and Murnaghan studied $H$-distinguished tame supercuspidal representations, for a symmetric subgroup $H$. They applied Mackey's theory to analyze the multiplicity of a tame supercuspidal representation, as a sum of multiplicities of representations of compact-mod-center subgroups. They studied which summands are non-zero, a property they termed \emph{strong compatibility} (see \cite[Definition 5.6]{hakim2008distinguished}).
They proved that strong compatibility implies a weaker condition, called \emph{moderate compatibility} (see \cite[Proposition 5.10]{hakim2008distinguished}), which has a simpler geometric description. We summarize their results in the following lemma.

\begin{proposition}\label{prop: cuspidal to compact}
    Let $\mathbf G$ be a tamely ramified connected reductive group over a $p$-adic field $F$ with residue characteristic greater than $\rank(\mathbf G)+1$. Let $\theta$ be a rational involution of $G= \mathbf G(F)$, $H$ the subgroup of $\theta$-fixed elements of $G$, and $\chi$ an unramified smooth character of $H$. Let $\rho$ be an irreducible $(H,\chi)$-distinguished supercuspidal representation of $G$.

    Then there exist
    \begin{itemize}
        \item an element $g_0\in G$;
        \item a $g_0.\theta$-stable compact-mod-center subgroup $K\subseteq G$;
        \item an irreducible $K$-representation $\pi$;
        \item a finite set $\cZ$ of $(K,H)$-double coset representatives
    \end{itemize}

    such that
    \begin{itemize}
        \item $\rho \cong \ind_K^G \pi$;
        \item for any element $\zeta \in \cZ$
              \[
                  \zeta g_0.\theta(\zeta )^{-1}\in K;
              \]
        \item the multiplicity of $\rho$ is equal to the following sum
              \[
                  \langle \rho, C^\infty(G/H,\chi) \rangle_G
                  =
                  \sum_{\zeta  \in \cZ}
                  \left\langle
                  \pi,\,
                  C^\infty\!\left(
                  K / K^{\zeta g_0.\theta}
                  \right)
                  \right\rangle_K .
              \]
    \end{itemize}

\end{proposition}
\begin{proof}
    By \cite[Theorem 8.1]{fintzen2021types} there exists a Yu datum $\Psi$ such that $\rho = \ind_K^G \pi$ for $K=K(\Psi)$ and $\pi = \kappa(\Psi)$ (see \cite[\S3.1 and \S3.4]{hakim2008distinguished}). By Mackey's decomposition formula we obtain
    \begin{align*}
        \langle \rho, C^\infty(G/H,\chi) \rangle_G
         & =
        \sum_{z  \in K\backslash G/H}
        \left\langle
        \pi,\,
        \chi\circ\opr{int_{z^{-1}}}
        \right\rangle_{K\cap\opr{int}_z(H)} \\
         & =
        \sum_{z  \in K\backslash G/H}
        \left\langle
        \pi,\,
        1
        \right\rangle_{K^{z.\theta}}        \\
         & =
        \sum_{z  \in K\backslash G/H}
        \left\langle
        \pi,\,
        C^\infty\!\left(
        K / K^{z.\theta}
        \right)
        \right\rangle_{K},
    \end{align*}
    where we used the fact that $\chi\circ\opr{int_{z^{-1}}}|_{K\cap\opr{int}_z(H)}=1$ as $\chi$ is unramified.
    As $\rho$ is $H$-distinguished there exists a $(K,H)$-double coset representative $g_0$ such that $\left\langle \pi,1\right\rangle_{K^{g_0.\theta}}\neq 0$.

    In \cite[\S5.1]{hakim2008distinguished} the authors define an equivalence relation on Yu data, called \emph{K-equivalence}, which has the following property
    \[
        \Psi \text{ is $K$-equivalent to } \dot{\Psi} \implies
        K \coloneqq K(\Psi) = K(\dot{\Psi}) \text{ and }
        \pi \coloneqq \kappa(\Psi) \cong \kappa(\dot{\Psi}).
    \]
    Denote the $K$-equivalence class of $\Psi$ by $\xi$.

    Denote the $K$-orbit of $g_0.\theta$ in the set of involutions on $G$ by $\Theta'$. By \cite[Definition 5.6]{hakim2008distinguished} $\Theta'$ and $\xi$ are strongly compatible.
    By \cite[Proposition 5.20]{hakim2008distinguished}, $\Theta'$ and $\xi$ are also moderately compatible (\cite[Definition 5.6]{hakim2008distinguished}). By \cite[Proposition 5.7]{hakim2008distinguished}, there is $\dot{\Psi}\in \xi$ which is $g_0.\theta$-symmetric (\cite[Definition 3.13]{hakim2008distinguished}), and by \cite[Proposition 3.14]{hakim2008distinguished}, we obtain $g_0.\theta(K)=K$.

    As $K$ contains the center of $G$, there is a bijection between $(K,H)$-double cosets and $K$-orbits in the $G$-orbit of $\theta$, defined by
    \[
        KzH \leftrightarrow Kz.\theta.
    \]
    By \cite[Proposition 5.10(2)]{hakim2008distinguished}, any other $(K,H)$-double coset whose reciprocal under the above bijection is strongly compatible with $\xi$, has a representative $\zeta$ such that $\zeta g_0.\theta(\zeta )^{-1}\in K^0\subseteq K$, where $K^0$ is defined in \cite[\S3.1]{hakim2008distinguished}.
\end{proof}

Using this result, Hakim and Murnaghan proved that the number of non-vanishing summands in Mackey's decomposition is finite. In the next section, we show this number is in fact bounded by a constant depending only on the absolute rank of $\mathbf G$ and the residue degree of $F$.
\subsection{From compact-mod-center to compact}\label{subsec: compact-mod-center to compact}
Propositions~\ref{prop: offen geometric lemma for cuspidal} and \ref{prop: cuspidal to compact} allow us to reduce to the setting of compact-mod-center subgroups. In this subsection, we explain how to further reduce to the compact subgroup setting.

We first construct a compact subgroup which, after multiplying by a maximal central torus, yields a subgroup of bounded index.
\begin{lemma}\label{lem: compact times center fin ind in compact-mod-center}
    Fix a positive integer $r$. There exists a constant $C = C(r)$ such that for every
    \begin{itemize}
        \item $p$-adic field $F$ with residue characteristic greater than $r+1$;
        \item connected reductive group $\mathbf{G}$ over $F$ with $\rank(\mathbf{G}) \leq r$;
        \item compact-mod-center subgroup $K \subseteq G \coloneqq \mathbf{G}(F)$ containing the center $Z(G)$;
    \end{itemize}
    the quotient $K / K'Z$ is a finite abelian group of order at most $C$. Here
    \[
        K' \coloneqq K \cap G_{\opr{der}}, \qquad G_{\opr{der}} \coloneqq \mathbf{G}_{\opr{der}}(F), \qquad Z \coloneqq \mathbf{Z}(F),
    \]
    where $\mathbf{G}_{\opr{der}}$ is the derived subgroup of $\mathbf{G}$ and $\mathbf{Z}$ is its maximal central torus.
\end{lemma}
The proof requires the following lemmas.

\begin{lemma}\label{lem: bound galois coho on finite modules}
    Let $F$ be a $p$-adic field with residue characteristic $p$, and let $\Gal(F)=\Gal(F^{\opr{sep}}/F)$ be the absolute Galois group of $F$. Let $A$ be a finite $\Gal(F)$-module of size $n$. If $p \nmid n$ then
    \[
        |H^1(\Gal(F),A)| \leq n^2.
    \]
\end{lemma}
\begin{proof}
    Let $F^{\opr{tame}}\subseteq F^{\opr{sep}}$ be the maximal tamely ramified extension of $F$, and let $P_F \coloneqq \Gal(F^{\opr{tame}})\subseteq \Gal(F)$ be the wild inertia group, which is a pro-$p$ group. By Lemma~\ref{lem: inf-res exact seq} and Lemma~\ref{lem: vanishing of l neq p coho} it is enough to show that
    \[
        |H^1(\Gal(F^{\opr{tame}}/F),B)| \leq m^2,
    \]
    for any $\Gal(F^{\opr{tame}}/F)$-module $B$ of size $m$ such that $p\nmid m$. This follows as $\Gal(F^{\opr{tame}}/F)$ is topologically generated by two elements, and each cocycle is determined by its values on these generators.
\end{proof}

\begin{lemma}\label{lem: size of central isogeny}
    Let $\mathbf{G}$ be a connected reductive group over a $p$-adic field $F$ with residue characteristic $p$.
    The map
    \[
        \mathbf m:\mathbf{Z} \times \mathbf{\mathbf{G}_{\opr{der}}} \to \mathbf{G}
    \]
    is a central isogeny.
    Moreover $\ker(\mathbf{m})$ is a finite (diagonalizable) algebraic group (in the sense of \cite[\S9]{Milne_reductive_groups}), and its order divides the order of $W_{\mathbf{G}}$, the absolute Weyl group of $\mathbf{G}$.
\end{lemma}

\begin{proof}
    By \cite[Theorem 3.2.2]{conrad2020reductive}, $\mathbf m$ is a central isogeny. For the second claim, we have a canonical identification
    \[
        \ker(\mathbf{m}) \cong \mathbf{Z}\cap\mathbf{\mathbf{G}_{\opr{der}}} \subseteq \mathbf{Z(\mathbf{G}_{\opr{der}})}.
    \]
    $\mathbf{\mathbf{G}_{\opr{der}}}$ is semi-simple and $W_{\mathbf{G}_{\opr{der}}} = W_{\mathbf{G}}$. Therefore it suffices to show that for any connected semi-simple group $\mathbf{G'}$ the order of $\mathbf{Z(G')}$ divides the order of $W_{\mathbf{G'}}$.

    Let $\mathbf{\widetilde{G'}}$ be the universal covering of $\mathbf{G'}$. As $\mathbf{Z(\widetilde{G'})}\twoheadrightarrow\mathbf{Z(G')}$ and $W_{\mathbf{G}'} = W_{\mathbf{\widetilde{G'}}}$, it is enough to show that $\mathbf{Z(\widetilde{G'})}$ divides the order of $W_{\mathbf{\widetilde{G'}}}$.

    Furthermore, we may assume $\widetilde{\mathbf{G}'}$ is split, as base changing to the separable closure preserves the order of $\ker(\mathbf{m})$ and $W_{\mathbf{\widetilde{G'}}}$. By \cite[Proposition 8.1.8]{springer1998linear},
    \[
        X^*(\mathbf{Z(\widetilde{G'})}) \cong X^*(\mathbf T)/\bZ\Phi
    \]
    where $X^*(-)\coloneqq \Hom(-,\bG_m)$ denotes the set of characters, $\mathbf T$ is a split maximal torus of $\mathbf{\widetilde{G'}}$, and $\Phi$ is the set of simple roots associated to $\mathbf{\widetilde{G'}}$. As $\mathbf{\widetilde{G'}}$ is simply connected, $X^*(\mathbf T)\cong P$, where $P$ is the weights lattice.

    There is a bijection between split simply connected semi-simple groups over a separably closed field and simply connected root data, and thus for root systems. Furthermore, it is enough to check the claim for the simple constituents of the root system. A case-by-case analysis of simple root data gives the required result.
\end{proof}
\begin{proof}[Proof of Lemma~\ref{lem: compact times center fin ind in compact-mod-center}]
    As $[K,K]\subseteq K'$, it is clear that $K/K'Z$ is abelian. Moreover, $K/K'Z$ is isomorphic to the image of $K$ under the map $G\mapsto G/(Z\cdot G_{\opr{der}})$. Therefore it is enough to bound the size of $|G/(Z\cdot G_{\opr{der}})|$.

    By Lemma~\ref{lem: size of central isogeny} and as $F$ is a perfect field, we have a short exact sequence
    \[
        1
        \to
        \ker(\mathbf m)(F^{\opr{sep}})
        \to
        \mathbf{Z}(F^{\opr{sep}}) \times \mathbf{G_{\opr{der}}}(F^{\opr{sep}})
        \to
        \mathbf{G}(F^{\opr{sep}})
        \to
        1,
    \]
    and
    \[
        |\ker(\mathbf m)(F^{\opr{sep}})| \text{ divides } |W_{\mathbf{G}}|.
    \]
    In particular, $|\ker(\mathbf m)(F^{\opr{sep}})|\leq |W_{\mathbf{G}}|$, and as $p>r+1 \geq\rank(\mathbf G)+1$, we also have $p\nmid |\ker(\mathbf m)(F^{\opr{sep}})|$.

    By Lemma~\ref{lem: exact sequence of coho}, we have
    \[
        |G/(Z\cdot G_{\opr{der}})| \leq |H^1(\Gal(F),\ker(\mathbf m)(F^{\opr{sep}}))|.
    \]
    By Lemma~\ref{lem: bound galois coho on finite modules} we have
    \[
        |H^1(\Gal(F),\ker(\mathbf m)(F^{\opr{sep}}))|\leq |\ker(\mathbf m)(F^{\opr{sep}})|^2 \leq |W_{\mathbf{G}}|^2.
    \]
    As the size of the absolute Weyl group is bounded by a constant depending only on $\rank(\mathbf{G})$, this finishes the proof.
\end{proof}

Next we show that $K'$ is compact.
\begin{lemma}\label{lem: K' is compact}
    $K'$ is compact.
\end{lemma}
\begin{proof}
    As $K'$ is a closed subgroup of $K$, its image under the map $p:K\to K/Z(G)$ is closed and hence compact. Moreover, $p(K')\cong K'/(K\cap G_{\opr{der}} \cap Z(G))$ so it is enough to show that $|G_{\opr{der}}\cap Z(G)|$ is finite. As $Z(G)/Z$ is finite, it is enough to show that $|G_{\opr{der}}\cap Z|$ is finite. This follows as $G_{\opr{der}}\cap Z$ embeds in $(\ker(\mathbf m))(F^{\opr{sep}})$ for $\mathbf m$ as in Lemma~\ref{lem: size of central isogeny}. By Lemma~\ref{lem: size of central isogeny}, $(\ker(\mathbf m))(F^{\opr{sep}})$ is finite.
\end{proof}

Let $\pi$ be an irreducible $K$-representation. We show how to bound the multiplicity of $\pi$ by a bounded sum of the multiplicity of an irreducible subrepresentation of $\pi|_{K'}$.

\begin{lemma}\label{lem: compact-mod-center to compact}
    There exists an irreducible $K'$-representation $\tau$ such that
    \[
        \langle \pi, C^\infty(K/K^\theta) \rangle_K
        \leq
        \sum_{\gamma \in K'Z\backslash K/K^\theta}
        \left\langle
        \tau,\,
        C^\infty\!\left(
        K' / (K')^{\gamma.\theta}
        \right)
        \right\rangle_{K'} .
    \]
\end{lemma}

We require the following auxiliary lemmas.
\begin{lemma}\label{lem: compact times center to compact}
    Let $\widetilde{\tau}$ be an irreducible representation of $K' Z$, and let $\theta'$ be an involution on $K'Z$. Then
    \[
        \langle \widetilde{\tau}, C^\infty(K'Z/(K'Z)^{\theta'}) \rangle_{K'Z}
        \leq
        \langle \widetilde{\tau}|_{K'}, C^\infty(K'/K'^{\theta'}) \rangle_{K'}.
    \]
\end{lemma}
\begin{proof}
    Define $\omega \coloneqq \widetilde{\tau}|_{Z}$ and
    \[
        C^\infty(K'Z/(K'Z)^{\theta'})_\omega
        \coloneqq
        \{
        f \in  C^\infty(K'Z/(K'Z)^{\theta'})
        \,|\,
        f(kz) = \omega(z)f(k)
        \}.
    \]
    We have
    \begin{align*}
        \langle \widetilde{\tau}, C^\infty(K'Z/(K'Z)^{\theta'}) \rangle_{K'Z}
         & =
        \langle \widetilde{\tau}, C^\infty(K'Z/(K'Z)^{\theta'})_\omega \rangle_{K'Z} \\
         & \leq
        \langle \widetilde{\tau}, C^\infty(K'Z/(K'Z)^{\theta'})_\omega \rangle_{K'}  \\
         & \leq
        \langle \widetilde{\tau}, C^\infty(K'/K'^{\theta'}) \rangle_{K'},
    \end{align*}
    where the last inequality holds as the map
    \[
        C^\infty(K'Z/(K'Z)^{\theta'})_\omega \to
        C^\infty(K'/K'^{\theta'}), \quad f\mapsto f|_{K'}
    \]
    is $K'$-equivariant and injective.
\end{proof}
\begin{lemma}\label{lem: reduction to compact is irrep}
    For any irreducible representation $\widetilde{\tau}$ of $K' Z$, $\widetilde{\tau}|_{K'}$ is irreducible.
\end{lemma}
\begin{proof}
    The multiplication map $K' \times Z \to K'Z$ is surjective, hence the pullback of $\widetilde{\tau}$ is an irreducible representation of $Z \times K'$.
    By \cite[Proposition 2.16]{bernstein1976representations}, this implies that the restriction $\widetilde{\tau}|_{K'}$ is irreducible.
\end{proof}

\begin{proof}[Proof of Lemma~\ref{lem: compact-mod-center to compact}]
    Note that $K'Z$ is normal in $K$, that the category of representations of $K'Z$ is semi-simple, and by Lemma~\ref{lem: compact times center fin ind in compact-mod-center}, $[K:K'Z]$ is finite. Therefore there is an irreducible sub-representation $\tilde{\tau}\subseteq \pi$ such that $\pi\hookrightarrow \ind_{K'Z}^K \tilde{\tau}$.
    Therefore
    \[
        \langle \pi, C^\infty(K/K^\theta) \rangle_K
        \leq
        \langle \ind_{K'Z}^K \widetilde{\tau}, C^\infty(K/K^\theta) \rangle_K.
    \]
    By Mackey's decomposition formula, we obtain
    \[
        \langle \ind_{K'Z}^K \widetilde{\tau}, C^\infty(K/K^\theta) \rangle_K
        = \sum_{\gamma \in K'Z \backslash K/K^\theta}
        \langle \widetilde{\tau}, 1 \rangle_{K'Z \cap \opr{int}_{\gamma}(K^\theta)},
    \]
    and by Frobenius reciprocity, we have
    \[
        \langle \widetilde{\tau}, 1 \rangle_{K'Z \cap \opr{int}_{\gamma}(K^\theta)}
        =
        \langle \widetilde{\tau}, C^\infty(K'Z/(K'Z)^{\gamma.\theta}) \rangle_{K'Z}
    \]

    As $K'$ and $Z$ are normal and $\theta$-stable, $K'Z$ is $\gamma.\theta$-stable. Therefore we may apply Lemma~\ref{lem: compact times center to compact}, and obtain
    \[
        \langle \widetilde{\tau}, C^\infty(K'Z/(K'Z)^{\gamma.\theta}) \rangle_{K'Z}
        \leq
        \langle \widetilde{\tau}|_{K'}, C^\infty(K'/K'^{\gamma.\theta}) \rangle_{K'}.
    \]
    Setting $\tau \coloneqq \widetilde{\tau}|_{K'}$, which is irreducible by Lemma~\ref{lem: reduction to compact is irrep}, we obtain the claimed inequality.
\end{proof}

\section{Double Coset Counts}\label{sec: double cosets counts}

Another key ingredient in the proof of Theorem~\ref{intthm: multiplicity bound unramified} is a bound on the number of double cosets $\cW$ and $\cZ$ that appear in Propositions~\ref{prop: offen geometric lemma for cuspidal} and \ref{prop: cuspidal to compact} respectively, whose corresponding summands are nonzero.
Note that we have already shown in Lemma~\ref{lem: compact times center fin ind in compact-mod-center} that the double cosets in Lemma~\ref{lem: compact-mod-center to compact} are bounded.

The structure of double coset spaces in the setting of symmetric pairs of reductive groups over $p$-adic fields has been extensively studied; see, for example, \cite{benoist2007polar, delorme2006analogue, helminck1993rationality, springer1988some}.

A particularly useful perspective, developed in \cite{delorme2006analogue}, connects these double coset spaces to the group cohomology set $H^1(C_2, -)$, where $C_2 = \langle \theta \rangle$ is the cyclic group of order two, and the coefficients are certain $\theta$-stable subgroups of $G$.

We begin by establishing a relation between the cohomology sets $H^1(C_2, -)$ and the geometry of double coset spaces. We then bound the size of the relevant cohomology set using results from \cite[§6]{aizenbud2024bounds}. Finally, we apply these bounds to the $(K, H)$ and $(P, H)$ double cosets that appear in our setting.

\subsection{Double cosets to \texorpdfstring{$H^1(C_2,-)$}{H1(S2,-)}}\label{subsec: fibers to coho}
For this section, it is helpful to spell out Definition~\ref{def: group cohomology} of group cohomology in the special case where the action is induced by an involution $\theta$, that is, the acting group is $\langle \theta \rangle\cong C_2$.

\begin{lemma}[cf. Definition~\ref{def: group cohomology}]\label{lem:H1}
    Let $G$ be a group equipped with an involution $\theta$. We may identify the set of $1$-cocycles
    \[
        Z^1(\langle\theta\rangle, G)
        \cong
        \{ x \in G \mid \theta(x) = x^{-1} \},
    \]
    by taking the image of $\theta$.
    The group $G$ acts on the $Z^1(\langle\theta\rangle, G)$ via twisted conjugation:
    \[
        g . x := g x \theta(g)^{-1}.
    \]
    Two cocycles are cohomologous if they are in the same orbit.

    In other words, $H^1(\langle \theta \rangle,G)$, the first cohomology set of $C_2$ in $G$, is canonically identified with the set of orbits of twisted conjugation on the set of $theta$-inverting elements in $G$.
\end{lemma}

Now, recall the symmetrization map $\iota : G \to G$, defined by $\iota(g) = g \theta(g)^{-1}$. For any subgroup $Q \subseteq G$, $\iota$ induces a map on double cosets:
\[
    \iota_Q : Q \backslash G / H \longrightarrow Q \backslash G / \theta(Q).
\]
The following lemma formalizes the key connection between double cosets and group cohomology of $C_2$.
\begin{lemma}
    \label{lem: fiber to cohomology}
    Let $Q \subseteq G$ be a subgroup and let $x \in G$. Then there is a canonical bijection:
    \[
        \iota_Q^{-1}(\iota_Q(QxH)) \leftrightarrow
        \ker\bigg(
        H^1\big(\langle x . \theta \rangle, Q(x)\big)
        \longrightarrow
        H^1\big(\langle x . \theta \rangle, G\big)
        \bigg),
    \]
    where $Q(x) := Q \cap x . \theta(Q)$ and $x.\theta(y) = x\theta(x^{-1}yx)x^{-1}$.
\end{lemma}

The proof of Lemma~\ref{lem: fiber to cohomology} proceeds in two steps. First, in Lemma~\ref{lem: identity fiber to cohomology}, we verify the statement in the special case $x = e$. Then, in Lemma~\ref{lem: fiber to identity fiber}, we reduce the general case to $x=e$.
\begin{lemma}
    \label{lem: identity fiber to cohomology}
    Let $Q \subseteq G$ be a subgroup. Then there is a canonical bijection:
    \[
        \iota_Q^{-1}(\iota_Q(QeH)) \leftrightarrow
        \ker\bigg(
        H^1\big(\langle \theta \rangle, Q(e)\big)
        \longrightarrow
        H^1\big(\langle \theta \rangle, G\big)
        \bigg),
    \]
    where $Q(e) := Q \cap \theta(Q)$.
\end{lemma}
\begin{proof}
    Applying Lemma~\ref{lem: exact sequence of coho} to $Q(e)\subseteq G$
    we obtain the exact sequence:
    \[
        1 \longrightarrow Q(e)^\theta \longrightarrow H \longrightarrow (Q(e) \backslash G)^\theta
        \longrightarrow H^1(\langle \theta \rangle, Q(e))
        \longrightarrow H^1(\langle \theta \rangle, G).
    \]
    This induces a bijection:
    \[
        (Q(e) \backslash G)^\theta / H \leftrightarrow
        \ker\left( H^1(\langle \theta \rangle, Q(e)) \longrightarrow H^1(\langle \theta \rangle, G) \right).
    \]

    Therefore, it is enough to show that $(Q(e) \backslash G)^\theta / H$ is naturally bijective to the fiber $\iota_Q^{-1}(\iota_Q(QeH)) \subseteq Q \backslash G / H$. In the first direction, we show that each $(Q,H)$-double coset in the fiber, contains a unique $(Q(e),H)$-double coset in $(Q(e) \backslash G)^\theta / H$.

    By definition, $QxH \in \iota_Q^{-1}(\iota_Q(QeH))$ if and only if there exist $q_1, q_2 \in Q$ such that:
    \[
        \iota(x) = q_1 \theta(q_2)^{-1} \in Q.
    \]
    Define $r := q_2^{-1} x \in QxH$. Then
    \[
        \iota(r) = q_2^{-1} \iota(x) \theta(q_2) = q_2^{-1} q_1 \in Q.
    \]
    Moreover, since $\theta(\iota(r)) = \iota(r)^{-1}$, we conclude that $\iota(r) \in Q(e)$, or equivalently, $r \in (Q(e) \backslash G)^\theta$.

    We show that $Q(e)rH$ is the unique $(Q(e), H)$-double coset in $QxH$ whose image under $\iota$ intersect $Q(e)$.
    Indeed, if $r' \in QrH = QxH$ also satisfies $\iota(r') \in Q(e)$, then for some $q \in Q$ and $h \in H$,
    \[
        r' = q r h
        \implies
        \iota(r') = q \iota(r) \theta(q)^{-1}
        \implies
        \theta(q) \in Q
        \implies
        q \in Q(e)
        \implies
        r' \in Q(e)rH.
    \]

    For the other direction, note that for any $r \in G$ such that $\iota(r) \in Q(e)$, the $(Q, H)$-double coset $QrH$ lies in $\iota_Q^{-1}(\iota_Q(QeH))$. Hence, the two sets are bijective, completing the proof.
\end{proof}

\begin{lemma}
    \label{lem: fiber to identity fiber}
    For any $x \in G$ and any subgroup $Q \subseteq G$, there is a canonical bijection:
    \[
        \iota_Q^{-1}(\iota_Q(QxH))
        \leftrightarrow
        (x . \iota_Q)^{-1}\big(x . \iota_Q(Qe G^{x.\theta})\big),
    \]
    where $G^{x.\theta}$ is the subgroup of elements fixed by $x.\theta$, and the map $x.\iota$ is defined by
    \[
        x.\iota(g) := g \cdot x.\theta(g^{-1}),
        \quad g\in G.
    \]
\end{lemma}
\begin{proof}
    Consider the following commutative diagram:
    \[
        \begin{tikzcd}
            G
            \arrow[r, "\iota"]
            \arrow[d, "\cdot x^{-1}"]
            & G \arrow[d, "\cdot \iota(x)^{-1}"] \\
            G
            \arrow[r, "x.\iota"]
            & G
        \end{tikzcd}
    \]
    Since $G^{x.\theta} = xHx^{-1}$, we obtain the following diagram on double cosets:
    \[
        \begin{tikzcd}
            Q \backslash G / H
            \arrow[r, "\iota_Q"]
            \arrow[d, "\cdot x^{-1}"]
            & Q \backslash G / \theta(Q)
            \arrow[d, "\cdot \iota(x)^{-1}"] \\
            Q \backslash G / G^{x.\theta}
            \arrow[r, "x.\iota_Q"]
            & Q \backslash G / x.\theta(Q)
        \end{tikzcd}
    \]

    Since the vertical maps are bijections, the fibers over $\iota_Q(QxH)$ and $\iota_Q(Qe \cdot G^{x.\theta})$ correspond via right multiplication by $x^{-1}$.

    Therefore,
    \[
        \iota_Q^{-1}(\iota_Q(QxH))
        = x.\iota_Q^{-1}(x.\iota_Q(Qe \cdot G^{x.\theta})),
    \]
    as claimed.
\end{proof}

\subsection{Bounds for \texorpdfstring{$H^1(C_2,-)$}{H1(S2,-)}}\label{subsec: bounds for H1(S2,K)}

The main result of this subsection is to establish a uniform bound on the size of the cohomology set $H^1(\langle \theta \rangle, K \cap \theta(K))$, where $K$ is a maximal compact-mod-center subgroup of the points of a connected reductive group $\mathbf{G}$ over a $p$-adic field.

\begin{lemma}
    \label{lem: bounds on coho of compact-mod-center}
    Fix positive integers $d,r$. There exists a constant $C = C(d,r)$ such that for every
    \begin{itemize}
        \item $p$-adic field $F$ with odd residue characteristic and residue degree $f_F \leq d$;
        \item connected reductive group $\mathbf G$ over $F$ with $\rank(\mathbf G) \leq r$;
        \item  compact-mod-center subgroup $K \subseteq G\coloneqq \mathbf G(F)$ containing the center $Z(G)$;
    \end{itemize}
    we have
    \[
        \left| H^1(\langle \theta \rangle, K \cap \theta(K)) \right| < C.
    \]
\end{lemma}
Lemma~\ref{lem: bounds on coho of compact-mod-center}, together with Lemma~\ref{lem: fiber to cohomology}, allows us to bound the number of double cosets contributing nontrivially to the multiplicity formulas of Propositions~\ref{prop: offen geometric lemma for cuspidal} and \ref{prop: cuspidal to compact}.

Our strategy is to apply Lemma~\ref{lem: exact sequence of coho} repeatedly, thereby reducing the cohomology computation to more manageable cases. Ultimately, we arrive at the setting of \cite[Proposition~6.0.5]{aizenbud2024bounds}, which provides an explicit bound on $H^1(\langle \theta \rangle, \Gamma)$ for any finite group $\Gamma$ admitting an embedding into $\GL_n(\bF_p)$; this bound depends only on $n$.

\subsubsection{Bounds for Quotients of Tori}
For the proof of Lemma~\ref{lem: bounds on coho of compact-mod-center} we require the following.

\begin{lemma}\label{lem: bound on coho of quotient of a torus}
    Let $F$ be a $p$-adic field with odd residue characteristic $p$, and let $S$ be the $F$-points of a torus $\mathbf{S}$ of rank $n$, equipped with an involution $\theta$. Let $\Gamma$ be a finite subgroup of $S$.
    Then
    \[
        |H^1(\langle \theta \rangle, S/\Gamma)| \leq |\Gamma|\cdot2^{3n}.
    \]
\end{lemma}

First we prove a few auxiliary lemmas.
\begin{lemma}\label{lem: abelian coho reduct to id and inv}
    Let $A$ be an abelian group equipped with an involution $\theta$. Then
    \[
        |H^1(\langle \theta \rangle,A)|\leq
        |A[2]|\cdot
        |A/A^2|,
    \]
    where $A[2] \coloneqq \{a\in A \,|\, a^2 = 1\}$.
\end{lemma}

\begin{proof}
    Define $\iota: A \to A$ by $\iota(a)=a\theta(a)^{-1}$ for all $a\in A$. Note that $\iota$ is a homomorphism as $A$ is abelian.  Set $A^{-\theta}\coloneqq \Im \iota$, and let $A^\theta\subseteq A$ be the subgroup of $\theta$-fixed points. Consider the $\theta$-stable exact sequence
    \[
        0\to A^\theta \to A \overset{\iota}{\to} A^{-\theta} \to 0.
    \]
    By Lemma~\ref{lem: exact sequence of coho} we have
    \[
        |H^1(\langle \theta \rangle,A)|\leq
        |H^1(\langle \theta \rangle,A^\theta)|
        |H^1(\langle \theta \rangle,A^{-\theta})|.
    \]

    In addition, we have
    \[
        |H^1(\langle \theta \rangle,A^\theta)|=
        |Z^1(\langle \id \rangle,A^\theta)|\leq
        |Z^1(\langle \id \rangle,A)|=
        |H^1(\langle \id \rangle,A)|.
    \]

    Note that
    \[
        Z^1(\langle \theta \rangle,A^{-\theta})
        =Z^1(\langle \opr{inv} \rangle,A^{-\theta})
        \subseteq Z^1(\langle \opr{inv} \rangle,A).
    \]
    Moreover, for any $z_1,z_2 \in Z^1(\langle \theta \rangle,A^{-\theta})$, for which there exists $a\in A^2$ such that $z_1 = az_2$, we have $a\in A^{-\theta}$. Hence
    \[
        |H^1(\langle \theta \rangle,A^{-\theta})|
        \leq |H^1(\langle \opr{inv} \rangle,A)|.
    \]

    Combining the two bounds above we get
    \[
        |H^1(\langle \theta \rangle,A)|\leq
        |H^1(\langle \id \rangle,A)|
        |H^1(\langle \opr{inv} \rangle,A)|.
    \]

    It is straightforward to verify that for any abelian group $A$ there are bijections
    \[
        H^1(\langle \id \rangle,A) \leftrightarrow A[2] ,\quad
        H^1(\langle \opr{inv} \rangle,A) \leftrightarrow A/A^2.
    \]
\end{proof}

For finite abelian groups we have the following corollary.

\begin{corollary}\label{cor: bound coho finite abelian}
    Let $A$ be a finite abelian group equipped with an involution $\theta$.  Then
    \[
        |H^1(\langle \theta \rangle,A)|\leq
        |A[2]|^2.
    \]
\end{corollary}
\begin{proof}
    By Burnside's lemma for the action of $A^2$ on $A$ by right multiplication, we have
    \[
        |A/A^2| = \frac{1}{|A|}\sum_{a\in A} |\opr{Fix}_{A^2}(a)| = |A[2]|.
    \]
\end{proof}
Now we are ready to prove  Lemma~\ref{lem: bound on coho of quotient of a torus}.

\begin{proof}[Proof of Lemma~\ref{lem: bound on coho of quotient of a torus}]
    By Lemma~\ref{lem: abelian coho reduct to id and inv} it is enough to show that
    \[
        |(S/\Gamma)/(S/\Gamma)^2| \leq 2^{2n}, \quad  |(S/\Gamma)[2]| \leq |\Gamma|\cdot2^{n}.
    \]

    We start by bounding $|(S/\Gamma)/(S/\Gamma)^2|$. Note that the quotient map $S\to S/\Gamma$ is surjective and compatible with squaring. Therefore we have a surjection
    \[
        S/S^2 \twoheadrightarrow (S/\Gamma)/(S/\Gamma)^2.
    \]
    As $S$ is of rank $n$, we have that $S$ is the fixed points of some action of $\Gal(F)$ on $\bG_m(F^{\opr{sep}})$. Consider the following exact sequence
    \[
        1
        \to
        (\mu_2(F^{\opr{sep}}))^n
        \to
        (\bG_m(F^{\opr{sep}}))^n
        \overset{(-)^2}{\to}
        (\bG_m(F^{\opr{sep}}))^n
        \to
        1,
    \]
    where $\mu_2$ is the group of roots of unity of order $2$, and  $F^{\opr{sep}}$ is the separable closure of $F$. Taking the fixed points of the Galois action of $\Gal(F)$ we get that
    \[
        |S/S^2| \leq |H^1(\Gal(F),(\mu_2(F^{\opr{sep}}))^n)|,
    \]
    where the action of $\Gal(F)$ on $(\mu_2(F^{\opr{sep}}))^n$ may be non-standard. As $p$ is odd we may apply Lemma~\ref{lem: bound galois coho on finite modules}, and obtain
    \[
        |H^1(\Gal(F),(\mu_2(F^{\opr{sep}}))^n)| \leq 2^{2n}.
    \]

    Now we bound $|(S/\Gamma)[2]|$. By unfolding the definitions we have
    \[
        (S/\Gamma)[2]
        = \{
        s\in S | s^2 \in \Gamma
        \}
        = \bigsqcup_{\gamma\in \Gamma}\{
        s\in S | s^2 = \gamma
        \}
        = \bigsqcup_{\gamma\in \Gamma\cap S^2} \sqrt{\gamma}S[2]
    \]
    where $\sqrt{\gamma}\in S$ is such that $(\sqrt{\gamma})^2=\gamma$. Therefore
    \[
        |(S/\Gamma)[2]| \leq |\Gamma| |S[2]|.
    \]

    Note that $S$ is a subgroup of $(F^{\opr{sep}})^{\times n}$, hence
    \[
        |S[2]| \leq |((F^{\opr{sep}})^{\times n})[2]| = 2^n.
    \]
    Multiplying the bounds proves the claim.
\end{proof}

\subsubsection{Bounds for Compact Subgroups}
We now turn to bound the size of the cohomology set $H^1(\langle \theta \rangle, K')$ for $\theta$-stable compact subgroups $K' \subset G$.
\begin{lemma}
    \label{lem: bounds on coho of compact}
    Fix positive integers $d,r$. There exists a constant $C = C(d,r)$ such that for every
    \begin{itemize}
        \item $p$-adic field $F$, with odd residue characteristic and residue degree $f_F \leq d$;
        \item connected reductive group $\mathbf G$ over $F$ with $\rank(\mathbf G) \leq r$;
        \item  $\theta$-stable compact subgroup $K' \subset G\coloneqq \mathbf G(F)$;
    \end{itemize}
    we have
    \[
        |H^1(\langle \theta \rangle, K')| < C.
    \]
\end{lemma}

To prove Lemma~\ref{lem: bounds on coho of compact}, we require the following two lemmas and definition. The first lemma provides an embedding of a finite quotient of $K'$ into a general linear group over the residue field, allowing us to apply the bound on the cohomology obtained in \cite[Proposition 6.0.5]{aizenbud2024bounds}.
\begin{lemma}
    \label{lem: rd_p of compact subgroups}
    Fix a positive integer $r$. There exists a constant $n = n(r)$ such that for every
    \begin{itemize}
        \item $p$-adic field $F$ with residue field $\ff$;
        \item connected reductive group $\mathbf G$ over $F$ with $\rank(\mathbf G) \leq r$;
        \item  compact subgroup $K' \subset G\coloneqq \mathbf G(F)$ with $\opr{Rad}_p(K') \trianglelefteq K'$ its pro-$p$ radical, that is, the maximal normal pro-$p$ subgroup;
    \end{itemize}
    we have
    \[
        K'/\opr{Rad}_p(K') \hookrightarrow \GL_n(\ff).
    \]
\end{lemma}
\begin{proof}
    Since $\mathbf G$ is a connected reductive group over $F$, there exists a faithful representation
    \[
        G \hookrightarrow \GL_n(F),
    \]
    for some $n$ depending only on $\rank(\mathbf G)$ (see \cite[Lemma 3.2.2]{aizenbud2024bounds}).

    By \cite[Theorem 3.1]{conradcompact}, the image of $K'$ under this embedding lies in a maximal compact subgroup of $\GL_n(F)$, which is conjugate to $\GL_n(\cO_F)$. We may assume, without loss of generality, that $K' \subset \GL_n(\cO_F)$.

    Let $\varpi$ be a uniformizer of $F$. We have the following exact sequence
    \[
        1 \to
        1 + \varpi \Mat_n(\cO_F) \to
        \GL_n(\cO_F) \to
        \GL_n(\ff) \to
        1.
    \]
    $1 + \varpi \Mat_n(\cO_F)$ is the pro-$p$ radical of $\GL_n(\cO_F)$, and its intersection with $K'$ is exactly $\opr{Rad}_p(K')$.

    Therefore, the composition
    \[
        K' \hookrightarrow \GL_n(\cO_F) \twoheadrightarrow \GL_n(\ff)
    \]
    factors through $K'/\opr{Rad}_p(K')$, yielding the desired embedding:
    \[
        K'/\opr{Rad}_p(K') \hookrightarrow \GL_n(\ff).
    \]
\end{proof}

\begin{lemma}
    \label{lem: vanishing of cohomology of pro-p groups}
    Let $p$ be an odd prime, and let $\Pi$ be a pro-$p$ group. Then for any involution $\theta$ on $\Pi$,
    \[
        |  H^1(\langle \theta \rangle, \Pi) | = 1.
    \]
\end{lemma}

\begin{proof}

    By \cite[§5.1]{serre1979galois}, there is a natural identification:
    \[
        H^1(\langle \theta \rangle, \Pi) \cong \varprojlim_{N \trianglelefteq_{\opr{open}} \Pi} H^1(\langle \theta \rangle, \Pi/N).
    \]
    Hence it is enough to show that
    \[
        |H^1(\langle \theta \rangle, \Gamma)| = 1
    \]
    for any finite $p$-group $\Gamma$ equipped with an involution $\theta$. Every finite $p$-group admits a central series whose successive quotients are isomorphic to $C_p$, the cyclic group of size $p$. As the center is stable under any involution, we may apply Lemma~\ref{lem: exact sequence of coho} iteratively, to reduce to the claim that
    \[
        |H^1(\langle \theta' \rangle, C_p)| = 1
    \]
    for any involution $\theta'$ on $C_p$.
    As $p$ is odd, $\langle \theta \rangle$ is isomorphic to the cyclic group of size $2$, the claim follows from Lemma~\ref{lem: vanishing of C_p coho}.
\end{proof}

\begin{definition}[cf. {\cite[Definition 2.2.1]{aizenbud2024bounds}}]
    \label{def: p-reductivity-dimension}
    Let $\Gamma$ be a finite group, and let $p$ be a prime number.

    \begin{itemize}
        \item The \emph{$p$-reductivity dimension} $\opr{rd}_p(\Gamma)$ is the minimal integer $n$ such that there exists a connected reductive algebraic group $\mathbf G$ over $\bF_p$ of dimension $n$ and an embedding
              \[
                  \Gamma \hookrightarrow \mathbf G(\bF_p).
              \]

        \item The \emph{reduced $p$-reductivity dimension} $\opr{rd}_p^{\opr{red}}(\Gamma)$ is defined as
              \[
                  \opr{rd}_p^{\opr{red}}(\Gamma) := \opr{rd}_p\big(\Gamma / \opr{Rad}_p(\Gamma)\big),
              \]
              where $\opr{Rad}_p(\Gamma)$ is the $p$-radical of $\Gamma$, that is, the maximal normal $p$-subgroup of $\Gamma$.
    \end{itemize}
\end{definition}

\begin{proof}[Proof of Lemma~\ref{lem: bounds on coho of compact}]
    Let $K' \subset G$ be a $\theta$-stable compact subgroup. Let $\opr{Rad}_p(K') \trianglelefteq K'$ be its pro-$p$ radical, which is $\theta$-stable. Since $p \neq 2$, Lemma~\ref{lem: vanishing of cohomology of pro-p groups} implies that
    \[
        | H^1(\langle \theta \rangle, \opr{Rad}_p(K')) |= 1.
    \]

    Applying Lemma~\ref{lem: exact sequence of coho} to the short exact sequence
    \[
        1 \to \opr{Rad}_p(K') \to K' \to K'/\opr{Rad}_p(K') \to 1,
    \]
    we obtain that
    \[
        |H^1(\langle \theta \rangle, K')| \leq |H^1(\langle \theta \rangle, K'/\opr{Rad}_p(K'))|.
    \]

    By Lemma~\ref{lem: rd_p of compact subgroups}, we have an embedding
    \[
        K'/\opr{Rad}_p(K') \hookrightarrow \GL_n(\ff),
    \]
    for some $n$ depending only on $\rank(\mathbf G)$. Since $\GL_n(\ff)$ embeds into $\GL_{n \cdot f_F}(\bF_p)$, where $f_F$ is the residue degree of $F$, we obtain
    \[
        \opr{rd}_p^{\opr{red}}(K'/\opr{Rad}_p(K')) \leq C,
    \]
    for some constant $C$ depending only on $\rank(\mathbf G)$ and $f_F$.

    By \cite[Proposition 6.0.5]{aizenbud2024bounds}, the size of $H^1(\langle \theta \rangle, K'/\opr{Rad}_p(K'))$ is bounded in terms of $\opr{rd}_p^{\opr{red}}(K'/\opr{Rad}_p(K'))$. Hence, the cohomology set $H^1(\langle \theta \rangle, K')$ is bounded by a constant depending only on $\rank(\mathbf G)$ and $f_F$.
\end{proof}

\subsubsection{Proof of Lemma~\ref{lem: bounds on coho of compact-mod-center}}
We are now ready to prove Lemma~\ref{lem: bounds on coho of compact-mod-center} which is the main result of \S\ref{subsec: bounds for H1(S2,K)}. We do so by combining the bound for compact subgroups established in Lemma~\ref{lem: bounds on coho of compact} and the bound for tori established in Lemma~\ref{lem: bound on coho of quotient of a torus}.

\begin{proof}[Proof of Lemma~\ref{lem: bounds on coho of compact-mod-center}]
    Set $K(e)\coloneqq K \cap \theta(K)$. By Lemma~\ref{lem: K' is compact}, $K(e)'\coloneqq K(e)\cap  G_{\opr{der}}$ is compact, and by Lemma~\ref{lem: compact times center fin ind in compact-mod-center} there exists a constant $C'=C'(r)$ such that $K(e)/(K(e)'Z)$ is an abelian group of size at most $C'$, where $Z$ is the maximal central torus of $G$. Note that both $K(e)'$ and $Z$ are $\theta$-stable.

    By Lemma~\ref{lem: exact sequence of coho} we have
    \[
        |H^1(\langle \theta \rangle, K(e))| \leq
        |H^1(\langle \theta \rangle, K(e)'Z)| \cdot
        |H^1(\langle \theta \rangle, K(e)/K(e)'Z)|.
    \]

    By Corollary~\ref{cor: bound coho finite abelian},
    \[
        |H^1(\langle \theta \rangle, K(e)/K(e)'Z)| \leq C'^2.
    \]
    Applying Lemma~\ref{lem: exact sequence of coho} a second time, we get
    \[
        |H^1(\langle \theta \rangle, K(e)'Z)| \leq
        |H^1(\langle \theta \rangle, K(e)')| \cdot
        |H^1(\langle \theta \rangle, K(e)'Z/K(e)')|.
    \]
    By Lemma~\ref{lem: bounds on coho of compact}, $|H^1(\langle \theta \rangle, K(e)')|$ is bounded by a constant depending on $r$. Note that
    \[
        K(e)'Z/K(e)'
        \cong Z/(Z\cap K(e)')
        \cong Z/(Z \cap K(e)\cap G_{\opr{der}})
        \cong Z/(Z\cap G_{\opr{der}}).
    \]

    By Lemma~\ref{lem: size of central isogeny}, $Z\cap G_{\opr{der}}$ embeds into $\ker(\mathbf m) (F^{\opr{sep}})$, whose size is bounded by a constant $C''=C''(r)$. By Lemma~\ref{lem: bound on coho of quotient of a torus} we have
    \[
        |H^1(\langle \theta \rangle, Z/(Z\cap G_{\opr{der}}))|
        \leq C'' \cdot 2^{3\rank(\mathbf{S})}
        \leq C'' \cdot 2^{3\rank(\mathbf{G})}.
    \]
\end{proof}
\begin{remark}
    \label{rem: better bounds for coho of compact}
    The bounds established in Lemma~\ref{lem: bounds on coho of compact} for compact subgroups — and consequently in Lemma~\ref{lem: bounds on coho of compact-mod-center} for compact-mod-center subgroups — are likely far from optimal. In particular, they depend on the residue degree $f_F$, a dependence which may well be avoidable.

    Nevertheless, since our final result (Theorem~\ref{thm: multiplicity bound unramified}) ultimately relies on \cite[Corollary B]{aizenbud2024bounds}, which itself exhibits this dependence, there is little incentive to remove it within the present framework.

    Still, it is plausible that a more refined analysis could eliminate the $f_F$ term. The dependence arises through \cite[Proposition 6.0.5]{aizenbud2024bounds}, which bounds $H^1(\langle \theta \rangle, \Gamma)$ for finite groups $\Gamma$ via embeddings into $\GL_n(\bF_p)$.

    One alternative is to exploit \cite[Lemma 6.0.1]{aizenbud2024bounds}, which provides bounds in terms of $\rank(\mathbf G)$ whenever $\Gamma$ is of the form $\mathbf G(\ff)$ for some connected reductive group $\mathbf G$ over $\ff$. A more detailed analysis using Bruhat--Tits theory may show that the finite groups occurring in our setting are always of this form — or at least sufficiently close to allow this refinement.

    Another possibility is to strengthen the arguments in \cite[\S6]{aizenbud2024bounds} directly so as to remove the $f_F$ dependence. A more careful application of results from \cite{larsen2011finite} may yield such an improvement. The authors of \cite{aizenbud2024bounds} did not pursue this, as their bounds already carried an independent dependence on $f_F$, arising from a different use of \cite{larsen2011finite}.
\end{remark}
\subsection{\texorpdfstring{$(K,H)$}{(K,H)} double cosets}\label{subsec: (K,H) double cosets}
In Proposition~\ref{prop: cuspidal to compact}, the only contributing double cosets in Mackey’s formula correspond to a single fiber of the symmetrization map \( \iota_K \). By Lemma~\ref{lem: fiber to cohomology}, the size of this fiber is bounded by the size of \( H^1(C_2, K) \). Finally, Lemma~\ref{lem: bounds on coho of compact-mod-center} bounds this cohomology uniformly, and we obtain the following consequence.
\begin{corollary}
    \label{cor: (K,H) double cosets bound}
    Fix positive integers $d,r$. There exist constants $C = C(d,r)$ such that for every
    \begin{itemize}
        \item $p$-adic field $F$ with odd residue characteristic and residue degree $f_F \leq d$;
        \item connected reductive group $\mathbf G$ over $F$, with $\rank(\mathbf G)\leq r$;
        \item  rational involution $\theta$ of $G\coloneqq \mathbf G(F)$  with $H = G^\theta$;
        \item  compact-mod-center subgroup $K \subseteq G$;
        \item $(K,\theta(K))$-double coset $\alpha \in K \backslash G / \theta(K)$;
    \end{itemize}
    we have
    \[
        |\left\{ x \in K \backslash G / H \mid \iota_K(x) = \alpha \right\}|<C.
    \]
\end{corollary}
\subsection{\texorpdfstring{$(P,H)$}{(P,H)} double cosets}\label{sec: (P,H) double cosets}

The finiteness of the number of $(P, H)$ double cosets was established in \cite[§6]{helminck1993rationality}, based on earlier work in \cite{springer1988some}. A refined argument appears in \cite[Proposition 3.2]{delorme2006analogue}, which uses group cohomology of $C_2$ to provide an explicit bound. Our approach is similar to \cite{delorme2006analogue}, but recasts it in the language developed in the previous subsections, and achieves an explicit bound.

\begin{lemma}
    \label{lem: (P,H) double cosets bound}
    Fix positive integers $d,r$. There exist constants $C = C(d,r)$ such that for every
    \begin{itemize}
        \item $p$-adic field $F$ with odd residue characteristic and residue degree $f_F \leq d$;
        \item connected reductive group $\mathbf G$ over $F$, with $\rank(\mathbf G)\leq r$;
        \item  rational involution $\theta$ of $G\coloneqq \mathbf G(F)$  with $H = G^\theta$;
        \item  standard parabolic subgroup $P \subseteq G$;
    \end{itemize}
    we have
    \[
        |P \backslash G / H|<C.
    \]
\end{lemma}
\begin{proof}
    We may assume without loss of generality that $P$ is the minimal standard parabolic $P_0$, as in \S\ref{subsec: p-adic Representations and Symmetric Spaces}\eqref{notation: standard}. Recall that we assumed that $P_0$ contains a fixed $\theta$-stable maximal $F$-split torus $S$.

    Let $N_G(S)$ be the normalizer and $Z_G(S)$ the centralizer of $S$ in $G$, and set $W\coloneqq N_G(S)/Z_G(S)$ be the Weyl group of $G$.
    By \cite[Lemma 6.8]{helminck1993rationality} we have
    \[
        |P_0 \backslash G / H| = |\iota_{Z_G(S)}^{-1}(Z_G(S)\backslash N_G(S)/ Z_G(S))|
    \]

    Note that
    \[
        |Z_G(S)\backslash N_G(S)/ Z_G(S)| = |W|
    \]
    which is bounded as a function of $\rank(\mathbf G)$.  Moreover, we have that $Z_G(S)$ is $w.\theta$-stable for any $w\in W$. Therefore, as $|W|$ is bounded, by Lemma~\ref{lem: fiber to cohomology}, it is enough to bound
    \begin{equation}
        \tag{$\star$}\label{eq: H1(w.theta,Z_G(S))}
        |H^1(\langle w.\theta \rangle,Z_G(S))|
    \end{equation}
    for any $w\in W$.

    As explained in \S\ref{subsec: p-adic Representations and Symmetric Spaces}\eqref{notation: standard}, $Z_G(S)=\mathbf{Z_G(S)}(F)$. Moreover, $Z_G(S)$ is compact modulo its center. Therefore we may apply Lemma~\ref{lem: bounds on coho of compact-mod-center}\footnote{Here $Z_G(S)$ plays the role of both the ambient group and the compact-mod-center subgroup.} to obtain a uniform bound on \eqref{eq: H1(w.theta,Z_G(S))} in terms of $\rank(\mathbf{Z_G(S)}) \leq \rank (\mathbf{G})$ and $f_F$.

    We conclude that
    \[
        |P\backslash G /H| \leq |W| \cdot \sup_{w\in W} \left| H^1(\langle w . \theta \rangle, Z_G(S)) \right|,
    \]
    and both terms on the right-hand side are uniformly bounded.
\end{proof}

\begin{remark}\label{rem: proof of prop: bounded fibers}
    This result is enough for the proof of Theorem~\ref{thm: multiplicity bound unramified}. In order to deduce Proposition~\ref{prop: bounded fibers}, note  that as
    \[
        |P\backslash G /H| = \sum_{y\in \iota_P(P\backslash G /H)} |\iota_P^{-1}(y)|,
    \]
    we also bound the size of each fiber.
\end{remark}

\begin{remark}
    \label{rem: M-admissible refinement}
    Again, the bound in Lemma~\ref{lem: (P,H) double cosets bound} is likely far from optimal.
    First, as noted in Remark~\ref{rem: better bounds for coho of compact}, the bounds on cohomology sets in Lemma~\ref{lem: bounds on coho of compact-mod-center} are themselves quite large.
    Second, in our specific application in Theorem~\ref{thm: multiplicity bound unramified},
    not all $(P,H)$-double cosets contribute to the multiplicity formula.

    As stated after Proposition~\ref{prop: offen geometric lemma for cuspidal},
    \cite{offen2017parabolic} shows that the relevant cosets satisfy
    \emph{$M$-admissibility} (see \cite[\S4]{offen2017parabolic}),
    and studies their structure in detail.
    A careful analysis---possibly exploiting the results of \cite{offen2017parabolic}
    on $M$-admissibility and the geometry of the relative Weyl group---may yield sharper bounds
    on the number of such cosets.
\end{remark}

\section{Proof of Main Results}\label{sec: proof of main results}
We are now ready to prove the main results of this paper, Theorems~\ref{thm: multiplicity bound unramified} and \ref{thm: global bound}.
\subsection{Bounds over local fields}\label{sec: bounds over local fields}

Theorem~\ref{thm: multiplicity bound unramified} is the core local result of this paper and it provides a substantial verification of \cite[Conjecture~C]{aizenbud2024bounds} in the setting of symmetric spaces.
\begin{theorem}
    \label{thm: multiplicity bound unramified}
    Fix positive integers $d,r$. There exists a constant $C = C(d,r)$ such that for every
    \begin{itemize}
        \item $p$-adic field $F$ with residue characteristic $p > r+1$ and residue degree at most $d$;
        \item tamely ramified connected reductive group $\mathbf G$ over $F$ with $\rank(\mathbf G) \le r$;
        \item rational involution $\theta$ of $G \coloneqq \mathbf G(F)$, with $H = G^\theta$;
        \item smooth admissible irreducible representation $\rho$ of $G$;
        \item smooth unramified character $\chi$ of $H$;
    \end{itemize}
    we have
    \[
        \left\langle
        \rho,\,
        C^\infty(G/H, \chi)
        \right\rangle_G
        < C.
    \]
\end{theorem}

\begin{proof}
    \textbf{Step 1: Reduction to the supercuspidal case.}
    By \cite[Proposition 3.19]{Bernstein1977}, there exists a standard Levi subgroup $M \coloneqq \mathbf M(F) \subseteq G$ and a supercuspidal representation $\sigma$ of $M$ such that $\rho \hookrightarrow i_{M,G}(\sigma)$. Applying Offen’s formula (Proposition~\ref{prop: offen geometric lemma for cuspidal}), we obtain
    \begin{equation}\label{eq: parbolic mackey}
        \left\langle
        i_{M,G}(\sigma),\
        C^\infty(G/H, \chi)
        \right\rangle_G
        =
        \sum_{\eta \in \cW}
        \left\langle
        \sigma,\
        C^\infty\!\left(M / M^{\eta.\theta}, \Delta_\eta \cdot \chi \circ \operatorname{int}_{\eta^{-1}} \right)
        \right\rangle_M,
    \end{equation}
    where $\cW$ is a finite set of $(P,H)$-double coset representatives. Since $|\cW|$ is bounded (Lemma~\ref{lem: (P,H) double cosets bound}), $M$ is tamely ramified (Lemma~\ref{lem: levi of tamely ramified is tamely ramified}), the character $\Delta_\eta \cdot \chi \circ \operatorname{int}_{\eta^{-1}}$ is unramified, and $\rank(\mathbf M) \le \rank(\mathbf G)$, it suffices to bound each summand. Thus, we may assume that $\rho$ is supercuspidal and $(H,\chi)$-distinguished.

    \textbf{Step 2: Reduction to compact-mod-center.}
    Assume $\rho$ is supercuspidal and $(H,\chi)$-distinguished. Since the residue characteristic satisfies $p > \rank(\mathbf G)+1$, we may apply Proposition~\ref{prop: cuspidal to compact} to obtain
    \begin{equation}\label{eq: compact mackey}
        \langle \rho, C^\infty(G/H,\chi) \rangle_G
        =
        \sum_{\zeta \in \cZ}
        \left\langle
        \pi,\,
        C^\infty\!\left(
        K / K^{\zeta g_0.\theta}
        \right)
        \right\rangle_K,
    \end{equation}
    where $K \subseteq G$ is a compact-mod-center subgroup, $\pi$ is an irreducible representation of $K$, $g_0 \in G$, and $\cZ$ is a finite set of $(K,H)$-double coset representatives. By Corollary~\ref{cor: (K,H) double cosets bound}, $|\cZ|$ is bounded, so it suffices to bound each summand. We therefore fix $K$, $\pi$, and a $K$-preserving rational involution $\theta'$.

    \textbf{Step 3: Reduction to the compact case.}
    Let $K' \coloneqq K \cap G_{\opr{der}}$. By Lemma~\ref{lem: compact-mod-center to compact}, there exists an irreducible representation $\tau$ of $K'$ such that
    \[
        \langle \pi, C^\infty(K/K^{\theta'}) \rangle_K
        \le
        \sum_{\gamma \in K'Z \backslash K / K^{\theta'}}
        \left\langle
        \tau,\,
        C^\infty\!\left(
        K' / (K')^{\gamma.\theta'}
        \right)
        \right\rangle_{K'}.
    \]
    Moreover, by Lemma~\ref{lem: compact times center fin ind in compact-mod-center}, the number of double cosets $|K'Z \backslash K|$ is bounded in terms of $\rank(\mathbf G)$.

    \textbf{Step 4: The compact case.}
    We conclude by applying \cite[Corollary B]{aizenbud2024bounds} to $\tau$, which yields a uniform bound depending only on $\rank(\mathbf G)$ and $f_F$. Although \cite[Corollary B]{aizenbud2024bounds} is stated for fields that are purely ramified extensions of $\bQ_p$, the argument extends by replacing the final embedding in the proof of \cite[Corollary B]{aizenbud2024bounds} with an embedding into $\GL_{f_F \cdot C^{\mathrm{lin}}(\rank(\mathbf G))}(\bF_p)$.
\end{proof}

\subsection{Global bounds}
Theorem~\ref{thm: multiplicity bound unramified} enables us to derive the following global result.

\begin{theorem}
    \label{thm: global bound}
    Fix positive integers $d,r$. There exist constants $C_1 = C_1(d,r)$ and $C_2 = C_2(d,r)$ such that for every
    \begin{itemize}
        \item number field $L$ with $[L : \bQ] \le d$;
        \item connected reductive group $\mathbf G$ over $L$ with $\rank(\mathbf G) \leq r$;
        \item finite place $v$ of $L$ such that the completion $L_v$ at $v$ has residue characteristic $p_v$ greater than $C_1$;
        \item rational involution $\theta$ of $\mathbf G$;
        \item irreducible smooth representation $\rho$ of $\mathbf G(L_v)$;
    \end{itemize}
    we have
    \[
        \langle \rho,\, C^\infty\!\left(\mathbf G(L_v)/\mathbf G(L_v)^{\theta_v}\right) \rangle_{\mathbf G(L_v)} < C_2.
    \]
\end{theorem}

\begin{proof}[Proof of Theorem~\ref{thm: global bound}]
    By Lemma~\ref{lem: bounded bad primes}, there exists a constant $C_1'$ such that for all finite places $v$ with $p_v > C_1'$, the group $G(L_v)$ is tamely ramified.

    Set $C_1 := \max\{C_1', r + 1\}$.
    By Theorem~\ref{thm: multiplicity bound unramified}, for each place $v$ with $p_v > C_1$ there exists a constant $C_2'(v)$, depending only on $\rank(\mathbf G_{L_v})$ and the residue degree $f_v$ of $L_v$, such that
    \[
        \langle \rho, C^\infty(G(L_v)/G(L_v)^{\theta_v}) \rangle < C_2'(v).
    \]
    Finally, by Lemmas~\ref{lem: residue field degree} and~\ref{lem: rank descent}, both $f_v$ and $\rank(\mathbf G_{L_v})$ are bounded in terms of $[L : \bQ]$ and $\rank(\mathbf G)$. Thus, $C_2$ can be chosen uniformly in $v$, completing the proof.
\end{proof}

\bibliographystyle{alpha}
\bibliography{references.bib}

\end{document}